\documentclass[a4paper]{article}
\usepackage[all]{xy}\usepackage[utf8]{inputenc}        
\usepackage[dvips]{graphics,graphicx}
\usepackage{empheq}
\usepackage{amsfonts,amssymb,amsmath,color,mathrsfs, amstext}
\usepackage{amsbsy, amsopn, amscd, amsxtra, amsthm,authblk}
\usepackage{enumerate,algorithmicx,algorithm}
\usepackage{algpseudocode}
\usepackage{upref}
\usepackage{geometry}
\usepackage[displaymath]{lineno}
\usepackage{float}
\usepackage{yhmath}
\usepackage{booktabs}
\usepackage{subfig} 
\usepackage{multirow}
\usepackage{makecell}
\usepackage[colorlinks,
            linkcolor=blue,
            anchorcolor=blue,
            citecolor=blue
            ]{hyperref}
\usepackage{cases}
\usepackage{algorithm}       
\usepackage{algpseudocode}   
\usepackage{natbib} 

\numberwithin{equation}{section}

\newcommand{\Rmnum}[1]{\uppercase\expandafter{\romannumeral #1}}

\newcommand{\keywords}[1]{\textbf{Keywords:} #1}
\newcommand{\MSC}[1]{\textbf{MSC2020:} #1}

\newtheorem{lem}{Lemma}[section]
\newtheorem{thm}{Theorem}[section]

\newtheorem{rmk}{Remark}[section]

\newtheorem{exm}{Example}[section]

\newtheorem{corly}{Corollary}[section]

\numberwithin{equation}{section}

\begin{document}

\title{Temporal Two-Grid Compact Difference Scheme for 
Benjamin-Bona-Mahony-Burgers Equation\thanks{The research of Dongling Wang is supported in part by the NSFC (No. 12271463), the 111 Project (No.D23017) and Program for Science and Technology Innovative Research Team in Higher Educational Institutions of Hunan Province of China.
 The work of Lisen Ding is supported by Postgraduate Scientific Research Innovation Project of Xiangtan University, China (No. XDCX2025Y213).}
}

\author[1]{Lisen Ding\thanks{E-mail: dingmath15@smail.xtu.edu.cn}}
\author[1]{Xiangyi Peng\thanks{E-mail: pxymath18@smail.xtu.edu.cn}}
\author[1]{Dongling Wang\thanks{E-mail: wdymath@xtu.edu.cn; Corresponding author.}}
\affil[1]{Hunan Key Laboratory for Computation and Simulation in Science and Engineering, School of Mathematics and Computational Science, Xiangtan University, Xiangtan, Hunan 411105, China.}

\maketitle

\begin{abstract}
This paper proposes a temporal two-grid compact difference (TTCD) scheme for solving the Benjamin-Bona-Mahony-Burgers (BBMB) equation with initial and periodic boundary conditions. The method consists of three main steps: first, solving a nonlinear system on a coarse time grid of size $\tau_c$; then obtaining a coarse approximation on the fine time grid of size $\tau_f$ via linear Lagrange interpolation; and finally solving a linearized scheme on the fine grid to obtain the corrected solution. The TTCD scheme reduces computational cost without sacrificing accuracy. Moreover, using the energy method, we rigorously prove the conservation property, unique solvability, convergence, and stability of the proposed scheme. It is shown that the method achieves convergence of order $\mathcal{O}(\tau_c^2 + \tau_f^2 + h^4)$ in the maximum norm,
where $h$ is space step size.
Finally, some numerical experiments are provided to demonstrate the effectiveness and feasibility of the proposed strategy.
\end{abstract}

\keywords{BBMB equation, temporal two-grid algorithm, compact finite difference, convergence and stability, conservation}

\MSC{65M06, 65M12, 65M55}

\section{Introduction}
Wave phenomena are widespread in nature, where most of them are nonlinear. Therefore, the study of nonlinear waves is a fundamental topic in physics, and the modeling tools are often nonlinear PDEs. As a model for describing shallow water wave propagation, the nonlinear Benjamin-Bona-Mahony (BBM) equation was first proposed in 1972 by Benjamin, Bona and Mahony \cite{Benjamin1972BBM}. Compared to the Korteweg-de Vries (KdV) equation, the authors proved that the BBM equation is better suited to simulating the unidirectional propagation of long waves with a small wave-amplitude. When the dissipative effect cannot be neglected, the dissipation term needs to be introduced into BBM equation, which results in the BBMB equation. 
In this article, we consider using temporal two-grid compact difference method to solve the 
nonlinear BBMB equation 
\begin{equation}\label{eq-bbmb} 
    u_t - \mu u_{xxt} +  uu_x +  u_x - \lambda u_{xx} = 0, \quad x\in\mathbb{R}, \ 0< t \leq T, 
\end{equation}
with initial-value and periodic boundary condition
\begin{align}
    &u(x,t) = u(x+L,t), \quad x\in\mathbb{R}, \quad 0< t \leq T, \label{eq-pdcond} \\
    &u(x,0) = \phi(x) , \quad x\in\mathbb{R}, \label{eq-initcond}
\end{align}
where $\mu$ and the viscosity coefficient $\lambda$ are positive constants.
Here, $u_{xxt}$ denotes the dispersion term, $u_{xx}$ represents the dissipation term, and $L$ is a spatial periodic length.


Through analytical approaches such as the exponential function method, variational iteration, and others, explicit solutions can be derived in specific cases for BBMB equation, as shown in \cite{Ganji2009Exp,Tari2007explicit,Bruzon2016Conservatio,besse2018artificial}. Nevertheless, these explicit solutions are often complex in form and limited to particular scenarios. For more general properties of the solution, numerical simulation becomes an essential and effective strategy to approximate and analyze the behavior of the BBMB equation.


Currently, a variety of numerical methods have been developed for solving the BBMB equation, such as the finite element method \cite{shi2024unconditional,ngondiep2024high}, finite difference method \cite{zhang2020numerical,cheng2021high,wang2025two}, pseudo-spectral or spectral element method \cite{mittal2018numerical,dehghan2021numerical}, virtual element method \cite{chen2025optimal}, meshless method \cite{dehghan2014numerical,shivanian2016more}, among others. Among these, developing high-order schemes has been a major research focus. For instance, Mohebbi and Faraz \cite{mohebbi2017solitary} proposed a standard fourth-order linearized difference scheme; Bayarassou \cite{bayarassou2021fourth} designed fourth-order nonlinear and linearized difference schemes; and Zhang and Liu \cite{zhang2021convergence} developed a linearized compact difference scheme. While such linearized schemes improve computational speed compared to the original nonlinear scheme, they often come at the expense of accuracy. On the other hand, nonlinear schemes must be solved iteratively. When the mesh size is very fine, solving the resulting large algebraic systems iteratively will inevitably leads to high computational cost. To balance computational expense with accuracy, the two-grid algorithm offers an effective strategy.

%

The two-grid algorithm was originally introduced by Xu \cite{xu1994novel,xu1996two}, with the primary aim of reducing the computational cost of solving large nonlinear systems while preserving the optimal accuracy of the method. Generally, its core idea involves first solving a small-scale nonlinear system on a coarse grid, and then using this coarse approximation to construct and solve a simpler linear system on a fine grid. Inspired by Xu's work, the two-grid approach has in recent years been successfully extended to a variety of nonlinear PDEs. For the space two-grid algorithm, we refer to \cite{hu2014two, chen2025spatial}, while details on the space-time two-grid method can be found in \cite{shi2024construction, gao2025efficient}.

In this paper, we focus on the time two-grid algorithm. For example, Liu et al. \cite{liu2018time} used a time two-grid finite element method to approximate a time-fractional water wave model with Burgers nonlinear term. Xu et al. \cite{xu2020time} applied a time two-grid finite difference approach to solve two-dimensional nonlinear fractional evolution equations. Peng et al. \cite{peng2024novel} developed a time two-grid compact difference scheme for the viscous Burgers equation. Additional details and further developments of the time two-grid method are also discussed in \cite{chai2023fast, chen2023two, niu2023fast}.

%
%
%
%
%

We develop this method and provide a corresponding theoretical analysis for the BBMB equation given in \eqref{eq-bbmb}-\eqref{eq-initcond}. The proposed scheme consists of three main steps. First, a nonlinear compact difference scheme is solved on a coarse grid using fixed-point iteration. Second, a coarse approximation on the fine grid is obtained via linear interpolation in time. Finally, using this coarse approximation, a linearized compact scheme is constructed and solved to obtain the corrected fine-grid solution. The main contributions of this paper can be summarized as follows.

\begin{itemize}
    \item 
A temporal two-grid compact difference (TTCD) scheme for the BBMB equation is established, and both the conservation property and uniqueness of the numerical solution are rigorously proved.

\item A comprehensive theoretical analysis based on the discrete energy method is presented. The scheme is shown to achieve a convergence order of 
$\mathcal{O}(\tau_c^2 + \tau_f^2 + h^4)$ in the maximum norm. Stability of the TTCD scheme follows directly from the same framework. 

\end{itemize}

The remainder of this paper is structured as follows. Section \ref{sec2} introduces necessary notations and fundamental lemmas. Section \ref{sec3} presents the 
TTCD scheme for the one-dimensional BBMB equation. The conservation property and unique solvability of the scheme are analyzed in Section \ref{sec4}. 
In Section \ref{sec6}, the convergence and stability of the proposed scheme are rigorously proved. Numerical experiments are provided in Section \ref{sec7} to validate the theoretical findings and demonstrate the efficiency of the TTCD scheme. Finally, a brief conclusion is drawn in Section \ref{sec8}.

\section{Some notations and lemmas} \label{sec2}
Before constructing the temporal two-grid algorithm, it is necessary to introduce some grid notation.

In the spatial direction, let the spatial step size be
 $h=\frac{L}{M}$ and define the spatial grid as $\Omega_h=\{x_p \mid x_p=a+ph,~a\in \mathbb{R},~1\leq p\leq M\}$ over one period. In the temporal direction, denote the coarse time step by $\tau_c=\frac{T}{N_c}$ and the coarse time grid by $\Omega_\tau^c=\{(t_c)_q \mid (t_c)_q=q\tau_c,~0\leq q\leq N_c\}$. Similarly, on the fine time grid, define the fine time step a $\tau_f=\frac{T}{N_f}$ and the fine time grid as  $\Omega_\tau^f=\{(t_f)_k \mid (t_f)_k=k\tau_f,~0\leq k \leq N_f\}$. 
Here, $M$, $N_c$ and $N_f$ are positive integers and $N_f=\beta_\tau N_c$, where $\beta_\tau \in \mathbb{Z}^{+}$ is referred to as the temporal step-size ratio. Clearly, $\tau_c=\beta_\tau \tau_f$. 
For each coarse grid point $\left(x_p,(t_c)_q\right)\in \Omega_h \times \Omega_{\tau}^{c}$, denote the corresponding coarse grid function as $w_p^q$. Similarly, the fine grid function is defined as  $w_p^k$ for any fine grid point $\left(x_p,(t_f)_k\right)\in \Omega_h \times \Omega_{\tau}^{f}$. For simplicity, we unify the notation $w_p^l = w_p^q$ or $w_p^k$, where the grid node $(x_p,t_l) \in \Omega_h \times \Omega_{\tau}^{\eta}$ $(0\leq l\leq N_{\eta})$ for $\eta = c$ or $f$. Then, for any grid functions $w_p^l, \, v_p^l$, we introduce 
\begin{equation*}
    \begin{aligned}
        &w^{l -\frac{1}{2}}_p = \frac{1}{2}\left( w^l_p +w^{l-1}_p\right), \quad& \delta_{t}^{\eta}& w^{l -\frac{1}{2}}_p = \frac{1}{\tau_{\eta}}\left( w^l_p - w^{l-1}_p \right),  \\
        \delta_x &w_{p-\frac{1}{2}}^l = \frac{1}{h}\left(w_{p}^l-w_{p-1}^l\right), \quad&  \delta_{xx}& w_{p}^l =\frac{1}{h}\left(\delta_x w_{p+\frac{1}{2}}^l - \delta_x w_{p-\frac{1}{2}}^l \right), \\
        \Delta_x &w_p^l = \frac{1}{2h}\left(w_{p+1}^l - w_{p-1}^l \right), \quad& \Psi&(v_p^l,w_p^l) =  \frac{1}{3}\left[ v_p^l\Delta_x w_p^l + \Delta_x \left(v_p^l w_p^l\right) \right].
    \end{aligned}
\end{equation*}
In a  periodic domain, the space of grid functions on $\Omega_h$ can be denoted as
\begin{equation*}
  \mathcal{W}_h = \left\{ w \mid  w =(w_1,w_2,\cdots,w_{M}), \;  w_{p} = w_{p+M} \ \text{for any} \ p\in\mathbb{Z} \right\}.
\end{equation*}
For any $v,w\in \mathcal{W}_h$, we define the following useful discrete inner products and the corresponding norms
\begin{equation*}
    \begin{aligned}
        &\langle v,w \rangle := h\sum_{p=1}^{M}v_p w_p, \quad (v,w) := h\sum_{p=1}^{M}\left( \delta_x v_{p-\frac{1}{2}}\right)\left(\delta_x w_{p-\frac{1}{2}}\right), \\
        &\Vert w \Vert := \sqrt{\langle w,w \rangle},\quad |w|_{1}:=\sqrt{(w,w)}, \quad \Vert w \Vert_{\infty} := \max_{1\leq p \leq M}|w_{p}|  . 
    \end{aligned}
\end{equation*}

According to the above notations, some crucial lemmas are showed as follows.  

\begin{lem}\textup{[\citenum{Sun2012}]}\label{lem-a}
  For any grid function $w\in \mathcal{W}_h$, it holds
    \begin{equation*}
        \Vert w \Vert_{\infty} \leq \frac{\sqrt{L}}{2} |w|_{1}, \quad |w|_{1} \leq \frac{2}{h} \Vert w \Vert, \quad  \Vert w \Vert \leq \frac{L}{\sqrt{6}}|w|_{1}, \quad \Vert \Delta_x w \Vert \leq |w|_{1}.
    \end{equation*}
\end{lem}

\begin{lem}\textup{[\citenum{Sun2012}]}\label{lem-b}
 For any grid functions $v,w\in \mathcal{W}_h$,  we have
  \begin{equation*}
      \langle v,\delta_{xx} w \rangle = - \langle \delta_x v,\delta_x w \rangle = \langle \delta_{xx} v, w \rangle,\quad  \langle \Delta_x v, w \rangle = -\langle  v,\Delta_x w \rangle.
  \end{equation*}
%
%
    \begin{equation*}
        \langle \Delta_x w, w \rangle = 0 ,\quad \langle \Delta_x w, \delta_{xx} w \rangle = 0 ,\quad \langle \Psi(v,w), w\rangle =0.
    \end{equation*}  
\end{lem}

\begin{lem}\textup{[\citenum{zhang2021convergence}]}\label{lem-g}
    Let $f(x)\in C^5[x_{p-1},x_{p+1}]$, and denote $F_p=f(x_p)$, $\widehat{F}_p=f^{\prime\prime}(x_p)$, then we have
    \begin{equation*}
        \begin{aligned}
            f(x_p)f^{\prime}(x_p) &= \Psi(F_p,F_p) - \frac{h^2}{2}\Psi(\widehat{F}_p,F_p) + \mathcal{O}(h^4), \\ 
            f^{\prime}(x_p) &= \Delta_x F_p - \frac{h^2}{6} \Delta_x \widehat{F}_p+ \mathcal{O}(h^4), \\ 
            f^{\prime\prime}(x_p) &= \delta_{xx} F_p - \frac{h^2}{12} \delta_{xx} \widehat{F}_p+ \mathcal{O}(h^4), 
        \end{aligned}
    \end{equation*}
    where $1\leq p\leq M$.
\end{lem}

\begin{lem}\textup{[\citenum{zhang2021convergence}]}\label{lem-d}
    For any grid functions $v,w,S\in\mathcal{W}_h$,  if
    \begin{equation*}
        w_p = \delta_{xx} v_p - \frac{h^2}{12}\delta_{xx} w_p + S_p, \quad 1\leq p \leq M,\\
    \end{equation*}
    then we obtain the following identities and inequalities,
    \begin{subequations}
        \begin{align}
            &\langle  w, v \rangle = -|v|_{1}^2 - \frac{h^2}{12}\Vert w \Vert^2 + \frac{h^4}{144} |w|_{1}^2 + \frac{h^2}{12}\langle  w, S \rangle + \langle  S, v \rangle, \label{eq-lemd1} \\
            &\langle  w, v \rangle \leq -|v|_{1}^2 - \frac{h^2}{18}\Vert w \Vert^2 + \frac{h^2}{12}\langle  w, S \rangle + \langle  S, v \rangle, \label{eq-lemd2}\\
            &\langle  \Delta_x w, v \rangle = \frac{h^2}{12}\langle \Delta_x w, S \rangle + \langle \Delta_x S, v \rangle. \label{eq-lemd3}
        \end{align}
    \end{subequations}
\end{lem}

\begin{lem}\textup{[\citenum{zhang2021convergence}]}\label{lem-e}
  For any grid function $v, w\in \mathcal{W}_h$,  we get
    \begin{equation*}
        \Delta_x (v_p w_p) = \frac{1}{2}w_{p+1} \left(\delta_x v_{p+\frac{1}{2}}\right)  + \frac{1}{2} w_{p-1} \left(\delta_x v_{p-\frac{1}{2}} \right)  +  v_p \Delta_x w_p,\quad 1\leq p \leq M, 
    \end{equation*}
    its vector form can be read as
    \begin{equation*}
        \Delta_x (v w) = \frac{1}{2}D_{+}w \left(\delta^{+\frac{1}{2}}_x v\right)  + \frac{1}{2} D_{-}w \left(\delta^{-\frac{1}{2}}_x v \right)  +  v \Delta_x w
    \end{equation*}
    where     
           $ D_{+}w_p=w_{p+1},  D_{-}w_p=w_{p-1}, \delta_x^{+\frac{1}{2}} v_p = \delta_x v_{p+\frac{1}{2}}, \delta_x^{-\frac{1}{2}} v_p = \delta_x v_{p-\frac{1}{2}}.$
\end{lem}

\begin{lem}\label{lem-f} 
  For any grid functions $v,w,S \in \mathcal{W}_h$, $\eta=c,f$, satisfying
    \begin{equation} 
        \begin{cases}
            w_p^l = \delta_{xx} v_p^l - \frac{h^2}{12}\delta_{xx} w_p^l + S_p^l, \quad 1\leq p \leq M, \ 0\leq l\leq N_{\eta},   \\
            v_p^l = v_{p+M}^l,\quad w_p^l = w_{p+M}^l,\quad  0\leq p\leq M, \ 0\leq l\leq N_{\eta}. 
        \end{cases}
    \end{equation}
    By \textup{[\citenum{wang2021pointwise}]}, we have
    \begin{equation}\label{eq-dltone}
        \begin{aligned}
            \langle w^{l-\frac{1}{2}}, \delta^{\eta}_t v^{l-\frac{1}{2}}  \rangle = &-\frac{1}{2\tau_{\eta}}\left[ \left( |v^l|_{1}^2 - |v^{l-1}|_{1}^2 \right) + \frac{h^2}{12}\left( \Vert w^l \Vert^2 - \Vert w^{l-1} \Vert^2\right) - \frac{h^4}{144}\left( |w^l|_{1}^2 - |w^{l-1}|_{1}^2 \right) \right] \\
            &+\frac{h^2}{12}\langle w^{l-\frac{1}{2}},\delta^{\eta}_t S^{l-\frac{1}{2}}\rangle + \langle S^{l-\frac{1}{2}}, \delta^{\eta}_t v^{l-\frac{1}{2}}  \rangle, \quad 1\leq l\leq N_{\eta}.
          \end{aligned}
    \end{equation}
    Similarly, we can also derive the following equality,
    \begin{equation}\label{eq-dlttwo}
        \begin{aligned}
            \langle \delta^{\eta}_t w^{l-\frac{1}{2}},  v^{l-\frac{1}{2}}  \rangle = &-\frac{1}{2\tau_{\eta}}\left[ \left( |v^l|_{1}^2 - |v^{l-1}|_{1}^2 \right) + \frac{h^2}{12}\left( \Vert w^l \Vert^2 - \Vert w^{l-1} \Vert^2\right) - \frac{h^4}{144}\left( |w^l|_{1}^2 - |w^{l-1}|_{1}^2 \right) \right] \\
            &+\frac{h^2}{12}\langle \delta^{\eta}_t w^{l-\frac{1}{2}}, S^{l-\frac{1}{2}} \rangle + \langle \delta^{\eta}_t S^{l-\frac{1}{2}}, v^{l-\frac{1}{2}}  \rangle, \quad 1\leq l\leq N_{\eta}.
          \end{aligned}
    \end{equation}
\end{lem}


\section{Construction of numerical scheme}\label{sec3}

In this section, we first construct a nonlinear compact difference (NCD) scheme for solving the BBMB equation \eqref{eq-bbmb}, employing the compact difference method for spatial discretization and the Crank-Nicolson method for temporal discretization.
Building upon this established scheme, we then design a 
TTCD scheme to enhance computational efficiency while maintaining the same grid resolution.

\subsection{NCD scheme}
Let $w=u_{xx}$, then \eqref{eq-bbmb} can be reformulated as 
\begin{subequations}\label{eq-pbm2}
    \begin{numcases}{}
        u_t - \mu w_{t} +  uu_x +  u_x - \lambda w = 0, \quad x\in\mathbb{R}, \ 0< t \leq T, \label{eq-bbmbnew} \\
        w=u_{xx}, \quad x\in\mathbb{R}, \ 0< t \leq T. \label{eq-twordfm}
    \end{numcases}
\end{subequations}
Next, we focus on the discretization of the problem \eqref{eq-bbmbnew}-\eqref{eq-twordfm} over a periodic domain $\Omega=(a,a+L]$, $a \in \mathbb{R}$. Define the grid functions $U_p^l=u(x_p,t_l)$ and $w_p^l=w(x_p,t_l)$ for $(x_p,t_l)\in \Omega_h \times \Omega_{\tau}^{\eta}$. By using Taylor expansion, $uu_x$, $u_x$, $u_{xx}$ can reach fourth-order approximation. Here, we omit the derivation details, and directly employ Lemma \ref{lem-g} to get 
\begin{equation}\label{eq-4orderterm}
  \left\{
  \begin{aligned}
    &w(x_p,t_l) = \delta_{xx} U_p^l - \frac{h^2}{12} \delta_{xx}  W_p^l+ \mathcal{O}(h^4),\\
    &uu_x(x_p,t_l) = \Psi(U_p^l,U_p^l) - \frac{h^2}{2}\Psi( W_p^l,U_p^l) + \mathcal{O}(h^4), \\
    &u_x(x_p,t_l) = \Delta_x U_p^l - \frac{h^2}{6} \Delta_x  W_p^l + \mathcal{O}(h^4),
  \end{aligned}
  \right.
\end{equation}
where $1<p<M$ and $0\leq l\leq N_{\eta}$, $\eta = c, f$.
Selecting $(x_p,t_{l-\frac{1}{2}})$, $(x_p,t_{l})$ as the stencil points for the equation \eqref{eq-bbmbnew} and \eqref{eq-twordfm}, respectively. Thus, leveraging \eqref{eq-4orderterm} and Crank-Nicolson scheme, we obtain 
\begin{subequations}\label{eq-pbm3}
    \begin{numcases}{}     
        \delta_t^{\eta} U_p^{l-\frac{1}{2}} - \mu\delta_t^{\eta} W_p^{l-\frac{1}{2}} +\Psi\left(U_p^{l-\frac{1}{2}},U_p^{l-\frac{1}{2}}\right) - \frac{h^2}{2}\Psi\left(W_p^{l-\frac{1}{2}},U_p^{l-\frac{1}{2}}\right)  \notag \\
         \quad+  \Delta_x U_p^{l-\frac{1}{2}} - \frac{h^2}{6}\Delta_x W_p^{l-\frac{1}{2}}  - \lambda W_p^{l-\frac{1}{2}} = (R_{\eta})_p^{l-\frac{1}{2}}, \quad 1\leq p\leq M, \ 1 \leq l \leq N_{\eta},  \label{eq-pbm3a}  \\
        W_p^{l} = \delta_{xx} U_p^l - \frac{h^2}{12}\delta_{xx} W_p^l + (Q_{\eta})_p^l, \quad 1\leq p\leq M, \ 0\leq l \leq N_{\eta}, \label{eq-pbm3b} 
    \end{numcases}
\end{subequations}
where the estimates of the truncation errors 
as follows:
\begin{equation}\label{eq-ErrEst}
    \begin{cases}
        | (R_{\eta})_p^{l-\frac{1}{2}} |   \leq C\left( \tau_{\eta}^2 + h^4 \right) , \quad &1\leq m\leq M, \ 1\leq l \leq N_{\eta},   \\
        | (Q_{\eta})_p^l | \leq C h^4 ,\quad &1\leq m\leq M, \ 0\leq l \leq N_{\eta}, \\
        | \delta^{\eta}_t (Q_{\eta})_p^{l-\frac{1}{2}} | \leq Ch^4 ,\quad &1\leq m\leq M, \ 1\leq l \leq N_{\eta}.
    \end{cases}
\end{equation}
In this paper, $C$ denotes a generic constant. Its value may change across different instances but remains independent of both the time step and space step sizes.

Omitting the high-order small terms in \eqref{eq-pbm3a}-\eqref{eq-pbm3b}, and replacing $U_p^l$, $W_p^l$ with $u_p^l$, $w_p^l$, then we derive the following NCD scheme of the problem \eqref{eq-bbmbnew}-\eqref{eq-twordfm} 
\begin{subequations}\label{eq-pbm4}
    \begin{numcases}{}     
        \delta_t^{\eta} u_p^{l-\frac{1}{2}} - \mu\delta_t^{\eta} w_p^{l-\frac{1}{2}} +\Psi\left(u_p^{l-\frac{1}{2}},u_p^{l-\frac{1}{2}}\right) - \frac{h^2}{2}\Psi\left(w_p^{l-\frac{1}{2}},u_p^{l-\frac{1}{2}}\right)   \notag \\
        \quad \quad +  \Delta_x u_p^{l-\frac{1}{2}} - \frac{h^2}{6}\Delta_x w_p^{l-\frac{1}{2}}  - \lambda w_p^{l-\frac{1}{2}} = 0, \quad 1\leq p\leq M, \ 1\leq l \leq N_{\eta}, \label{eq-pbm4a}\\
        w_p^{l} = \delta_{xx} u_p^l - \frac{h^2}{12}\delta_{xx} w_p^l , \quad 1\leq p\leq M, \ 0\leq l \leq N_{\eta}, \label{eq-pbm4b} \\
        u_p^l = u_{p+M}^l, \quad w_p^l = w_{p+M}^l, \quad 1\leq p\leq M, \ 0\leq l \leq N_{\eta}, \label{eq-pbm4c} \\
        u_p^0=\phi\left(x_p\right),\quad 1\leq p\leq M. \label{eq-pbm4d}
    \end{numcases}
\end{subequations}

\begin{rmk}\label{rmk-1}
Note that the index $\eta$ controls the type of time grid. In order to get the NCD scheme on coarse and fine time grid, we have made the following changes to the grid functions 
\begin{equation*}
    \{U_p^l,\, W_p^l\}\Longrightarrow\left\{
    \begin{aligned}
        &\{U_p^q, \, W_p^q\},\quad \text{for $\eta=c$}, \\
        &\{U_p^k, \, W_p^k\},\quad \text{for $\eta=f$}.
    \end{aligned}
    \right.
    ,\quad
    \{u_p^l,\, w_p^l\}\Longrightarrow\left\{
    \begin{aligned}
        &\{(u_c)_p^q, \, (w_c)_p^q\},\quad \text{for $\eta=c$}, \\
        &\{(u_f)_p^k, \, (w_f)_p^k\},\quad \text{for $\eta=f$}.
    \end{aligned}
    \right.
\end{equation*}
\end{rmk}

\begin{rmk}
  The above NCD scheme \eqref{eq-pbm4a}-\eqref{eq-pbm4d} is a nonlinear system, here we choose  fixed-point iteration to solve it. In actual calculation, the large $M$ and iteration result in an extremely high computational cost.
\end{rmk}

\subsection{TTCD scheme}
The TTCD scheme can be summarized in three steps, which are described in detail below.

\textbf{Step \Rmnum{1}.} Let $\eta=c$ in the system \eqref{eq-pbm4a}-\eqref{eq-pbm4d}, then the NCD scheme on the coarse grid $\Omega_h \times \Omega_{\tau}^c$ can be represented as follows:
\begin{subequations}\label{eq-pbm5}
    \begin{numcases}{}     
        \delta^c_t (u_c)_p^{q-\frac{1}{2}} - \mu\delta^c_t (w_c)_p^{q-\frac{1}{2}} +\Psi\left((u_c)_p^{q-\frac{1}{2}},(u_c)_p^{q-\frac{1}{2}}\right) - \frac{h^2}{2}\Psi\left((w_c)_p^{q-\frac{1}{2}},(u_c)_p^{q-\frac{1}{2}}\right)   \notag \\
        \quad \quad +  \Delta_x (u_c)_p^{q-\frac{1}{2}} - \frac{h^2}{6}\Delta_x (w_c)_p^{q-\frac{1}{2}}  - \lambda (w_c)_p^{q-\frac{1}{2}} = 0, \quad 1\leq p\leq M, \ 1\leq q \leq N_c, \label{eq-pbm5a}\\
        (w_c)_p^{q} = \delta_{xx} (u_c)_p^q - \frac{h^2}{12}\delta_{xx} (w_c)_p^q , \quad 1\leq p\leq M, \ 0\leq q \leq N_c, \label{eq-pbm5b} \\
        (u_c)_p^q = (u_c)_{p+M}^q, \quad (w_c)_p^q = (w_c)_{p+M}^q, \quad 1\leq p\leq M, \ 0\leq q \leq N_c, \label{eq-pbm5c} \\
        (u_c)_p^0=\phi\left(x_p\right),\quad 1\leq p\leq M. \label{eq-pbm5d}
    \end{numcases}
\end{subequations}

\textbf{Step \Rmnum{2}.} Considering the fine grid $\Omega_h \times \Omega_\tau^f$. 
%
Based on the temporal grid ratio $\beta_{\tau}$, we have $(t_f)_{q\beta{\tau}}=(t_c)_q$ for $1\leq q\leq N_{c}$.
For a fixed spatial node $x_{p}$, we perform linear Lagrange interpolation in the temporal direction
\begin{equation}\label{eq-timeintepo}
    \begin{aligned}
    (u_f)_{p}^{(q-1)\beta_{\tau}+ r} 
    &= \frac{(t_c)_{q} - (t_f)_{(q-1)\beta_{\tau}+ r}}{(t_c)_{q} - (t_c)_{q-1}}(u_c)_{p}^{q-1} + \frac{(t_f)_{(q-1)\beta_{\tau}+ r} - (t_c)_{q-1}}{(t_c)_{q} - (t_c)_{q-1}}(u_c)_{p}^{q} \\
    &= \left(1-\frac{r}{\beta_{\tau}}\right)(u_c)_{p}^{q-1} + \frac{r}{\beta_{\tau}}(u_c)_{p}^{q}, \quad   1\leq q\leq N_{c}, \quad 1\leq r\leq \beta_{\tau}-1. 
    \end{aligned}
\end{equation}
Note that $(u_f)_p^{q\beta_{\tau}} = (u_c)_p^{q}$ ($1\leq p\leq M$, $1\leq q\leq N_{c}$), then the values of $(u_f)^{k}_{p}$ ($0\leq k\leq N_f$) can be yielded from Step \Rmnum{1} and \eqref{eq-timeintepo}. 
When $\eta=f$, $(w_f)_p^k$ can be calculated with the aid of \eqref{eq-pbm4b}, i.e.,
\begin{equation}\label{eq-stwintepo}
    (w_f)_p^{k} = \delta_{xx} (u_f)_p^k - \frac{h^2}{12}\delta_{xx} (w_f)_p^k , \quad 1\leq p\leq M, \ 0\leq k \leq N_{f}.
\end{equation}

\textbf{Step \Rmnum{3}.} Since the accuracy of the approximate solutions $(u_f)_p^k$ remains insufficient, we proceed to design the following linearized scheme to further correct:
\begin{subequations}\label{eq-pbm6}
    \begin{numcases}{}     
        \delta^f_t u_p^{k-\frac{1}{2}} - \mu\delta^f_t w_p^{k-\frac{1}{2}} +\Psi\left((u_f)_p^{k-\frac{1}{2}},u_p^{k-\frac{1}{2}}\right) - \frac{h^2}{2}\Psi\left((w_f)_p^{k-\frac{1}{2}},u_p^{k-\frac{1}{2}}\right)   \notag \\
        \quad \quad +  \Delta_x u_p^{k-\frac{1}{2}} - \frac{h^2}{6}\Delta_x w_p^{k-\frac{1}{2}}  - \lambda w_p^{k-\frac{1}{2}} = 0, \quad 1\leq p\leq M, \ 1\leq k \leq N_f, \label{eq-pbm6a}\\
        w_p^{k} = \delta_{xx} u_p^k - \frac{h^2}{12}\delta_{xx} w_p^k , \quad 1\leq p\leq M, \ 0\leq k \leq N_f, \label{eq-pbm6b} \\
        u_p^k = u_{p+M}^k, \quad w_p^k = w_{p+M}^k, \quad 1\leq p\leq M, \ 0\leq k \leq N_f, \label{eq-pbm6c} \\
        u_p^0=\phi\left(x_p\right),\quad 1\leq p\leq M. \label{eq-pbm6d}
    \end{numcases}
\end{subequations}
Compared with the rough approximations $(u_f)_p^k$, the corrected solutions $u_p^k$ achieve higher accuracy.

\section{Conservation and unique solvability}\label{sec4}
\subsection{Conservation}
In a periodic domain, the BBMB equation \eqref{eq-bbmb} contains a conservative invariant. Defining the continuous energy functional 
\begin{equation}\label{eq-contienergy}
 E(t) = \int_{y}^{y+L} u^2(x,t) + \mu u_x^2(x,t) \, \mathrm{d}x + 2\lambda \int_{0}^{t}\int_{y}^{y+L} u_x^2(x,t) \, \mathrm{d}x \, \mathrm{d}t, \quad \forall y\in \mathbb{R}.
\end{equation}
We know that $E(t)  = E(0)$ from \cite{shi2024unconditional}. Below, we shall verify that the TTCD scheme can preserve discrete energy invariant.

\begin{thm}\label{thm-2}
    Let $\{(u_c)_p^{q}, (w_c)_p^{q} \ | \ 1\leq p\leq M, 0\leq q\leq N_c \}$ be the solutions of 
    the Step \Rmnum{1} \eqref{eq-pbm5a}-\eqref{eq-pbm5d}. Define
    \begin{equation}
        \begin{aligned}
        \mathring{E}^q = \Vert (u_c)^q \Vert^2 &+ \mu \left( |(u_c)^q|_1^2 + \frac{h^2}{12}\Vert (w_c)^q \Vert^2 - \frac{h^4}{144} | (w_c)^q |_1^2 \right) \\
        &+ 2\lambda\tau_c \sum_{n=1}^{q}\left( |(u_c)^{n-\frac{1}{2}}|_1^2 + \frac{h^2}{12}\Vert (w_c)^{n-\frac{1}{2}}\Vert^2 - \frac{h^4}{144}| (w_c)^{n-\frac{1}{2}} |_1^2 \right), \quad 0\leq q\leq N_c.
        \end{aligned}
    \end{equation}
   Then  it holds that $ \mathring{E}^q = \mathring{E}^0$  for any $0\leq q\leq N_c.$
\end{thm}

\begin{proof}
Taking an inner product of \eqref{eq-pbm4a} with $(u_c)^{q-\frac{1}{2}}$, we get
\begin{equation}\label{eq-inproderg}
    \begin{aligned}
        &\left\langle \delta^c_t  (u_c)^{q-\frac{1}{2}}, (u_c)^{q-\frac{1}{2}}\right\rangle - \mu \left\langle \delta^c_t (w_c)^{q-\frac{1}{2}}, (u_c)^{q-\frac{1}{2}}\right\rangle + \left\langle\Psi\left((u_c)^{q-\frac{1}{2}},(u_c)^{q-\frac{1}{2}}\right) , (u_c)^{q-\frac{1}{2}}\right\rangle \\ 
        & \quad - \frac{  h^2}{2}\left\langle\Psi\left((w_c)^{q-\frac{1}{2}},(u_c)^{q-\frac{1}{2}}\right) , (u_c)^{q-\frac{1}{2}}\right\rangle 
        +\left\langle \Delta_x (u_c)^{q-\frac{1}{2}}, (u_c)^{q-\frac{1}{2}} \right\rangle- \frac{ h^2}{6}\left\langle\Delta_x (w_c)^{q-\frac{1}{2}}, (u_c)^{q-\frac{1}{2}} \right\rangle \\ 
        &\quad - \lambda \left\langle (w_c)^{q-\frac{1}{2}}, (u_c)^{q-\frac{1}{2}} \right\rangle = 0, \quad 1\leq q\leq N_c.
    \end{aligned}
\end{equation}
By Lemma \ref{lem-b} and taking $S=0$ in \eqref{eq-lemd3} of Lemma \ref{lem-d}, we have
\begin{equation*}
    \left\langle\Psi\left(~\cdot~,(u_c)^{q-\frac{1}{2}}\right) , (u_c)^{q-\frac{1}{2}}\right\rangle = 0,\quad \left\langle \Delta_x (u_c)^{q-\frac{1}{2}}, (u_c)^{q-\frac{1}{2}} \right\rangle=0,\quad \left\langle\Delta_x (w_c)^{q-\frac{1}{2}}, (u_c)^{q-\frac{1}{2}} \right\rangle=0.
\end{equation*}

Regarding the remaining inner product terms, we proceed to analyze them step by step. It is straightforward to see that
\begin{equation*}
    \left\langle \delta^c_t  (u_c)^{q-\frac{1}{2}}, (u_c)^{q-\frac{1}{2}}\right\rangle = \frac{1}{2\tau_c}\left( \Vert (u_c)^{q} \Vert^2 - \Vert (u_c)^{q-1} \Vert^2 \right).
\end{equation*}
Let $S=0$ in \eqref{eq-dlttwo} of Lemma \ref{lem-f} and using \eqref{eq-lemd1} of Lemma \ref{lem-d}, we can yield
\begin{equation*}
    \begin{aligned}
    - \mu \left\langle \delta^c_t (w_c)^{q-\frac{1}{2}}, (u_c)^{q-\frac{1}{2}}\right\rangle = &\frac{\mu}{2\tau_c}\bigg[ \left( |(u_c)^q|_{1}^2 - |(u_c)^{q-1}|_{1}^2 \right) + \frac{h^2}{12}\left( \Vert (w_c)^q \Vert^2 - \Vert (w_c)^{q-1} \Vert^2\right) \\
    &- \frac{h^4}{144}\left( |(w_c)^q|_{1}^2 - |(w_c)^{q-1}|_{1}^2 \right) \bigg], 
    \end{aligned}
\end{equation*}
and 
\begin{equation*}
    - \lambda \left\langle (w_c)^{q-\frac{1}{2}}, (u_c)^{q-\frac{1}{2}} \right\rangle = \lambda \left(|(u_c)^{q-\frac{1}{2}}|_{1}^2 + \frac{h^2}{12}\Vert (w_c)^{q-\frac{1}{2}} \Vert^2 - \frac{h^4}{144} |(w_c)^{q-\frac{1}{2}}|_{1}^2\right).
\end{equation*}
Substituting the above results into \eqref{eq-inproderg}, and changing upper index $q$ to $n$. Finally,  summing $n$ from $1$ to $q$, thereby we obtain
\begin{equation*}
    \begin{aligned}
    \frac{1}{2\tau_c}\left( \Vert (u_c)^{q} \Vert^2 - \Vert (u_c)^0 \Vert^2 \right) &+ \frac{\mu}{2\tau_c}\left[ \left( |(u_c)^q|_{1}^2 - |(u_c)^0|_{1}^2 \right) + \frac{h^2}{12}\left( \Vert (w_c)^q \Vert^2 - \Vert (w_c)^0 \Vert^2\right) - \frac{h^4}{144}\left( |(w_c)^q|_{1}^2 - |(w_c)^0|_{1}^2 \right) \right] \\
    &+\lambda \sum_{n=1}^{q}\left(|(u_c)^{n-\frac{1}{2}}|_{1}^2 + \frac{h^2}{12}\Vert (w_c)^{n-\frac{1}{2}} \Vert^2 - \frac{h^4}{144} |(w_c)^{n-\frac{1}{2}}|_{1}^2\right)=0,
    \end{aligned}
\end{equation*}  
which yields the results.
\end{proof}

\begin{corly}
    For the Step \Rmnum{3}, a conservative invariant is also possessed by the linearized scheme \eqref{eq-pbm5a}-\eqref{eq-pbm5d}. Denote 
    \begin{equation}\label{eq-numinvariant}
        \begin{aligned}
        E^k = \Vert u^k \Vert^2 &+ \mu \left( |u^k|_1^2 + \frac{h^2}{12}\Vert w^k \Vert^2 - \frac{h^4}{144} | w^k |_1^2 \right) \\
        &+ 2\lambda\tau_f \sum_{n=1}^{k}\left( |u^{k-\frac{1}{2}}|_1^2 + \frac{h^2}{12}\Vert w^{n-\frac{1}{2}}\Vert^2 - \frac{h^4}{144}| w^{n-\frac{1}{2}} |_1^2 \right), \quad 0\leq k\leq N_f,
        \end{aligned}
    \end{equation}
    we can easily derive that $E^k = E^0$, where its proof is almost the same as Theorem \ref{thm-2}.
\end{corly}

\begin{rmk}
    Benefiting from the inequality in Lemma \ref{lem-a}, we have 
    \begin{equation*}
        \mathring{E}^q \geq  \Vert (u_c)^q \Vert^2 + \mu \left( |(u_c)^q|_1^2 + \frac{h^2}{18}\Vert (w_c)^q \Vert^2 \right)+ 2\lambda\tau_c \sum_{n=1}^{q}\left( |(u_c)^{n-\frac{1}{2}}|_1^2 + \frac{h^2}{18}\Vert (w_c)^{n-\frac{1}{2}}\Vert^2 \right) \geq  \Vert (u_c)^q \Vert^2 \geq 0.
    \end{equation*}
    It implies that the boundedness of $\Vert (u_c)^q \Vert$, i.e., there exists a constant $C$ such that
    \begin{equation}\label{eq-numsoltbnd}
        \Vert (u_c)^q \Vert \leq C, \quad 0\leq q\leq N_c.
    \end{equation}
    Similarly, we also have $\Vert u^k \Vert \leq C$ for $0\leq k\leq N_f$.
\end{rmk}

\subsection{Unique solvability}\label{sec5}
In what follows, we shall prove the unique solvability of solutions of TTCD scheme. Since the Step \Rmnum{2} only involves interpolation, its unique solvability is clear. Thus, we only need to consider the unique solvability of Steps \Rmnum{1} and \Rmnum{3}.

\begin{thm}
    The Step \Rmnum{1} (NCD) \eqref{eq-pbm5a}-\eqref{eq-pbm5d} admits a unique solution.  
\end{thm}

\begin{proof}
    we first prove the existence of solution.
    Denote $(u_c)_{p}:=(u_c)_{p}^{q-\frac{1}{2}}$, $(w_c)_{p}:=(w_c)_{p}^{q-\frac{1}{2}}$. It is easy to know that

    \begin{equation}\label{eq-newuw}
        (u_c)^{q}_{p} = 2(u_c)_{p}-(u_c)^{q-1}_{p},\quad  (w_c)^{q}_{p} = 2(w_c)_{p}-(w_c)^{q-1}_{p},\quad 1\leq p\leq M.
    \end{equation}
    By means of the equality relation \eqref{eq-pbm5b} and \eqref{eq-pbm5d}, we can determine the values of $(u_c)^0$ and $(w_c)^0$. Furthermore, suppose $(u_c)^{q-1},(w_c)^{q-1}$ are known, 
    and substitute the equality \eqref{eq-newuw} into \eqref{eq-pbm5a}-\eqref{eq-pbm5d}, then we get a new system about $u_c$ and $w_c$ as follows:
    \begin{equation}\label{eq-pbmuni1}
        \left\{
        \begin{aligned}
        &\frac{2}{\tau_{c}}\left((u_c)_p - (u_c)^{q-1}_p \right) - \frac{2\mu}{\tau_{c}}\left((w_c)_p - (w_c)^{q-1}_p \right) + \Psi\left((u_c)_p,(u_c)_p\right) - \frac{h^2}{2}\Psi\left((w_c)_p,(u_c)_p\right)   \\
        &\quad + \Delta_x (u_c)_p - \frac{h^2}{6}\Delta_x (w_c)_p - \lambda (w_c)_p=0, \quad 1\leq p\leq M,  \\ 
        &(w_c)_p = \delta_{xx} (u_c)_p - \frac{h^2}{12}\delta_{xx} (w_c)_p,\quad 1\leq p\leq M, \\ 
        &(u_c)_p = (u_c)_{p+M}, \quad (w_c)_p = (w_c)_{p+M}, \quad 1\leq p\leq M.
        \end{aligned}
        \right.
    \end{equation}
    Next, we define the nonlinear operator $\mathscr{F}(u_c):\mathcal{W}_h \to \mathcal{W}_h$:
    \begin{equation*}
        \begin{aligned}
            \mathscr{F}\left((u_c)_p\right)=&
            \frac{2}{\tau_{c}}\left((u_c)_p - (u_c)^{q-1}_p \right) - \frac{2\mu}{\tau_{c}}\left((w_c)_p - (w_c)^{q-1}_p \right) + \Psi\left((u_c)_p,(u_c)_p\right) - \frac{h^2}{2}\Psi\left((w_c)_p,(u_c)_p\right) \\ 
            &+  \Delta_x (u_c)_p - \frac{h^2}{6}\Delta_x (w_c)_p - \lambda (w_c)_p, \quad 1\leq p\leq M, 
        \end{aligned}
    \end{equation*} 
    and take an inner product of $\mathscr{F}(u_c)$ with $u_c$ as follows: 
    \begin{equation*}
        \begin{aligned}
            \langle \mathscr{F}\left(u_c\right),  u_c \rangle &= \frac{2}{\tau_{c}}\left( \Vert u_c \Vert^2 - \langle (u_c)^{q-1},  u_c \rangle \right) - \frac{2\mu}{\tau_{c}} \langle w_c -  (w_c)^{q-1},  u_c \rangle - \lambda \langle w_c,u_c\rangle \\ 
            &+ \langle \Psi\left(u_c,u_c\right),  u_c \rangle - \frac{h^2}{2}\langle \Psi\left(w_c,u_c\right),  u_c \rangle  + \langle \Delta_x u_c,u_c\rangle - \frac{h^2}{6} \langle\Delta_x w_c,u_c\rangle.  \\
        \end{aligned}
    \end{equation*}
    
    According to the Browder theorem [Lemma 3.1, \citenum{akrivis1991fully}], if there exists a positive constant $\alpha$ such that for every $u_c \in\mathcal{W}_h$ with $\Vert u_c \Vert=\alpha$, it holds $\langle \mathscr{F}(u_c),  u_c \rangle\geq 0$, then the system \eqref{eq-pbmuni1} possesses at least one solution.
    By the Lemma \ref{lem-b}, and let $S=0$ in \eqref{eq-lemd2} of Lemma \ref{lem-d}, we yield
    \begin{equation*}
        \begin{aligned}
            \langle \mathscr{F}(u_c),  u_c \rangle &\geq \frac{2}{\tau_{c}}\left[ \Vert u_c \Vert^2 - \Vert (u_c)^{q-1} \Vert \cdot \Vert u_c \Vert \right] + \frac{2\mu}{\tau_{c}}\left( |u_c|_{1}^2 + \frac{h^2}{18}\Vert w_c-(w_c)^{q-1} \Vert^2 \right) + \lambda \left( |u_c|_{1}^2 + \frac{h^2}{18}\Vert w_c \Vert^2 \right) \\
            &\geq \frac{2}{\tau_{c}}\left[ \Vert u_c \Vert (\Vert u_c \Vert - \Vert (u_c)^{q-1} \Vert) \right].
        \end{aligned}
    \end{equation*}
    when $\alpha = \Vert (u_c)^{q-1} \Vert$, then for any $\Vert u_c \Vert=\alpha$, we have $\langle \mathscr{F}(u_c), u_c \rangle \geq 0$. It means that the system NCD scheme \eqref{eq-pbm5a}-\eqref{eq-pbm5d} has at least one solution.

    Assume the system \eqref{eq-pbmuni1} exists two different solutions, then we can easily prove the uniqueness of solution by combining the energy method with contradiction. The details of uniqueness can refer to [Theorem 3.4, \citenum{wang2021pointwise}].  This completes the proof.
\end{proof}

\begin{thm}
The Step \Rmnum{3} \eqref{eq-pbm6a}-\eqref{eq-pbm6d} admits a unique solution.
\end{thm}

\begin{proof}
    Firstly, the $u^0$ and $w^0$ can be computed by applying \eqref{eq-pbm6b} and \eqref{eq-pbm6d}. Moreover, suppose $u^{k-1},w^{k-1}$ are known, then the problem \eqref{eq-pbm6a}-\eqref{eq-pbm6d} is a nonhomogeneous linear system about $u^k$, where the coefficient matrix is a $M \times M$ matrix. Now, considering the following corresponding homogeneous system,
    \begin{subequations}
        \begin{numcases}{}
            \frac{1}{\tau_f}u_p^{k} - \frac{\mu}{\tau_f}w_p^{k} + \frac{1}{2}\Psi\left((u_f)_p^{k-\frac{1}{2}},u_p^k\right)- \frac{h^2}{4}\Psi\left((w_f)_p^{k-\frac{1}{2}},u_p^k\right) \notag  \\
            \quad +  \frac{1}{2} \Delta_x u_p^k - \frac{h^2}{12}\Delta_x w_p^k - \frac{\lambda}{2} w_p^k = 0, \quad 1\leq p\leq M, \label{eq-Lpbm1a} \\
            w_p^{k} = \delta_{xx} u_p^k - \frac{h^2}{12}\delta_{xx} w_p^k , \quad 1\leq p\leq M. \label{eq-Lpbm1b}
        \end{numcases}
    \end{subequations}
Taking an inner product of \eqref{eq-Lpbm1a} with $u^k$, and using Lemma \ref{lem-b}, we have
\begin{equation*}
    \frac{1}{\tau_f}\Vert u^k \Vert^2 - \left(\frac{\mu}{\tau_f}+\frac{\lambda}{2}\right)\langle w^k,u^k \rangle-\frac{h^2}{12}\langle \Delta_x w^k,u^k \rangle = 0.
\end{equation*}
Let $S=0$ in \eqref{eq-lemd2} and \eqref{eq-lemd3} of Lemma \ref{lem-d}, then we get
\begin{equation*}
    \frac{1}{\tau_f}\Vert u^k \Vert^2 + \left(\frac{\mu}{\tau_f}+\frac{\lambda}{2}\right)\left( |u^k|_{1}^2  + \frac{h^2}{18}\Vert w^k \Vert^2\right)\leq 0,
\end{equation*}
it mean that $u^k=0$ and $w^k=0$. Thus, the corresponding nonhomogeneous linear system admits a unique solution, and the proof is completed.
\end{proof}

\section{Convergence and stability analysis}\label{sec6}
This section is divided into three parts to analyze the error estimates of the TTCD scheme in the maximum norm. Finally, the stability result of the TTCD scheme is presented.

\subsection{Error estimate for Step \Rmnum{1}}
Now we proceed to prove the convergence of the scheme \eqref{eq-pbm5a}-\eqref{eq-pbm5d}.
Define the following notation:
\begin{equation*}
(e_c)_p^q = U_p^q - (u_c)_p^q, \quad (\rho_c)_p^q = W_p^q-(w_c)_p^q , \quad 1\leq p\leq M, \ 0\leq q \leq N_c,
\end{equation*}
and the constant
\begin{equation}\label{eq-notan2}
C_0 = \max_{(x,t)\in \mathbb{R}\times [0,T]}{{|u(x,t)|,|u_x(x,t)|,|u_{xx}(x,t)|,|u_{xxx}(x,t)|}}.
\end{equation}

Taking $\eta = c$ in scheme \eqref{eq-pbm3a}-\eqref{eq-pbm3b}, and subtracting the scheme \eqref{eq-pbm5a}-\eqref{eq-pbm5d} from it, we obtain the error system
\begin{subequations}\label{eq-errpbm1}
\begin{numcases}{}
    \delta^{c}_t (e_c)_p^{q-\frac{1}{2}} - \mu\delta^{c}_t (\rho_c)_p^{q-\frac{1}{2}}  + \Psi\left(U_p^{q-\frac{1}{2}},U_p^{q-\frac{1}{2}}\right) - \Psi\left((u_c)_p^{q-\frac{1}{2}},(u_c)_p^{q-\frac{1}{2}}\right) \notag \\
    \quad - \frac{h^2}{2} \left[\Psi\left(W_p^{q-\frac{1}{2}},U_p^{q-\frac{1}{2}}\right) - \Psi\left((w_c)_p^{q-\frac{1}{2}},(u_c)_p^{q-\frac{1}{2}}\right)\right] + \Delta_x (e_c)_p^{q-\frac{1}{2}} - \frac{h^2}{6}\Delta_x (\rho_c)_p^{q-\frac{1}{2}}  \notag \\
    \quad - \lambda (\rho_c)_p^{q-\frac{1}{2}}  = (R_c)_p^{q-\frac{1}{2}}, \quad 1\leq p\leq M, \quad 0\leq q\leq N_{c},  \label{eq-errpbm1a}\\
    (\rho_c)_p^q = \delta_{xx} (e_c)_p^q - \frac{h^2}{12}\delta_{xx} (\rho_c)_p^q + (Q_c)_p^q,\quad 1\leq p\leq M, \quad 0\leq q\leq N_{c}, \label{eq-errpbm1b} \\
    (e_c)_p^q = (e_c)_{p+M}^q, \quad (\rho_c)_p^q = (\rho_c)_{p+M}^q, \quad 1\leq p\leq M, \quad 0\leq q\leq N_{c}. \label{eq-errpbm1c} 
\end{numcases}
\end{subequations}

\begin{thm}\label{thm-1} 
       Suppose that $\{u(x,t), w(x,t) \mid (x,t) \in \Omega \times [0,T] \}$ is the exact solution of problem \eqref{eq-bbmbnew}-\eqref{eq-twordfm}, and let $\{(u_c)_p^q, (w_c)p^q \mid 1 \leq p \leq M, \; 0 \leq q \leq N_c \}$ be the numerical solution of the coarse-grid problem \eqref{eq-pbm5a}-\eqref{eq-pbm5d}.
If $2C_1 \tau_c < 1/3$ and $\tau_c^2 + h^4 \leq 1/C$, where $C_1$ is defined in \eqref{eq-twoconsta}, then we have
\begin{equation}\label{eq-cvg1}
\lvert (e_c)^q \rvert_{1} \leq C (\tau_c^2 + h^4), \qquad 0 \leq q \leq N_c.
\end{equation}
\end{thm}   

\begin{proof}
If $q=0$, \eqref{eq-cvg1} is obviously valid, since the initial value is given. By taking an inner product of $(\rho_c)^0$ with \eqref{eq-errpbm1b}, and utilizing Lemmas \ref{lem-a} and \ref{lem-b}, we have 
$$
    \Vert (\rho_c)^0 \Vert^2 = - \frac{h^2}{12}\left\langle \delta_{xx} (\rho_c)^0, (\rho_c)^0 \right\rangle + \left\langle (Q_c)^0, (\rho_c)^0 \right\rangle ~ \Longrightarrow ~ \Vert (\rho_c)^0 \Vert^2 \leq  \frac{9}{4}\Vert (Q_c)^0 \Vert^2.
$$
Then, matching the estimates \eqref{eq-ErrEst} for $\eta=c$ , we get
\begin{equation}\label{eq-rhoes0}
    \left\Vert (\rho_c)^0 \right\Vert^2 \leq \frac{9}{4}L\left(Ch^4\right)^2.
\end{equation}
Now, we assume \eqref{eq-cvg1} holds for  $0\leq q \leq n-1$ $(1 \leq n \leq N_c)$. When $\tau_c^2+h^4\leq 1/C$
and using Lemma \ref{lem-a}, we have 
\begin{equation}\label{eq-errall1}
    \left|(e_c)^{q}\right|_{1} \leq 1, \quad \left\Vert (e_c)^{q} \right\Vert \leq \frac{L}{\sqrt{6}}, \quad \left\Vert (e_c)^{q} \right\Vert_{\infty} \leq \frac{\sqrt{L}}{2}, \quad  0\leq q\leq n-1.
\end{equation}

In what follows, we begin to verify that \eqref{eq-cvg1} holds for $q=n$. Taking an inner product of \eqref{eq-errpbm1a} with $\delta^c_t (e_c)^{q-\frac{1}{2}}$, we yield
\begin{equation}\label{eq-inproduct2}
    \begin{aligned}
        &\Vert \delta^c_t (e_c)^{q-\frac{1}{2}} \Vert^2 - \mu \left\langle \delta^c_t(\rho_c)^{q-\frac{1}{2}} , \delta^c_t (e_c)^{q-\frac{1}{2}} \right\rangle 
        +  \left\langle\Psi\left(U^{q-\frac{1}{2}},U^{q-\frac{1}{2}}\right) - \Psi\left((u_c)^{q-\frac{1}{2}},(u_c)^{q-\frac{1}{2}}\right), \delta^c_t (e_c)^{q-\frac{1}{2}} \right\rangle \\
        &- \frac{ h^2}{2} \left\langle \Psi\left(W^{q-\frac{1}{2}},U^{q-\frac{1}{2}}\right) - \Psi\left((w_c)^{q-\frac{1}{2}},(u_c)^{q-\frac{1}{2}}\right), \delta^c_t (e_c)^{q-\frac{1}{2}}\right\rangle - \lambda \left\langle (\rho_c)^{q-\frac{1}{2}},\delta^c_t (e_c)^{q-\frac{1}{2}} \right\rangle \\
        &+\left\langle \Delta_x (e_c)^{q-\frac{1}{2}}, \delta^c_t (e_c)^{q-\frac{1}{2}} \right\rangle -\frac{h^2}{6} \left\langle \Delta_x (\rho_c)^{q-\frac{1}{2}}, \delta^c_t (e_c)^{q-\frac{1}{2}}\right\rangle = \left\langle (R_c)^{q-\frac{1}{2}},  \delta^c_t (e_c)^{q-\frac{1}{2}}\right\rangle, \, 1\leq q\leq n.
    \end{aligned}
\end{equation}

Firstly, we estimate the term containing the operator $\Delta_x $. By the Cauchy-Schwarz inequality and Young's inequality, we have 
\begin{equation}\label{eq-alphaes1}
    \begin{aligned}
        -\left\langle \Delta_x (e_c)^{q-\frac{1}{2}}, \delta^c_t (e_c)^{q-\frac{1}{2}} \right\rangle 
        &\leq \frac{1}{8}\Vert \delta^c_t (e_c)^{q-\frac{1}{2}} \Vert^2 + 2|(e_c)^{q-\frac{1}{2}}|_{1}^2, \\
        \frac{h^2}{6}\left\langle \Delta_x (\rho_c)^{q-\frac{1}{2}}, \delta^c_t (e_c)^{q-\frac{1}{2}} \right\rangle &\leq \frac{1}{8}\Vert \delta^c_t (e_c)^{q-\frac{1}{2}} \Vert^2 + \frac{h^4}{18}|(\rho_c)^{q-\frac{1}{2}}|_{1}^2.
    \end{aligned}
\end{equation}

Secondly, according to \eqref{eq-lemd2} in Lemma \ref{lem-d}, we get
\begin{equation*}
    \begin{aligned}
    \mu\left\langle \delta^c_t(\rho_c)^{q-\frac{1}{2}} , \delta^c_t (e_c)^{q-\frac{1}{2}} \right\rangle \leq& -\mu|\delta^c_t (e_c)^{q-\frac{1}{2}}|_{1}^2 -\frac{\mu h^2}{18}\Vert \delta^c_t(\rho_c)^{q-\frac{1}{2}} \Vert^2  \\
    &+ \frac{\mu h^2}{12}\left\langle \delta^c_t (\rho_c)^{q-\frac{1}{2}}, \delta^c_t (Q_c)^{q-\frac{1}{2}} \right\rangle + \mu\left\langle \delta^c_t (Q_c)^{q-\frac{1}{2}}, \delta^c_t (e_c)^{q-\frac{1}{2}} \right\rangle.
    \end{aligned}
\end{equation*}
By applying Lemma \ref{lem-a}, Cauchy-Schwarz inequality and Young's inequality, the results as follows: 
\begin{equation*}
    \begin{aligned}
        -&\mu|\delta^c_t (e_c)^{q-\frac{1}{2}}|_{1}^2 \leq - \frac{6\mu}{L^2}\Vert \delta^c_t (e_c)^{q-\frac{1}{2}} \Vert^2 ,\\
        &\mu\left\langle \delta^c_t (Q_c)^{q-\frac{1}{2}}, \delta^c_t (e_c)^{q-\frac{1}{2}} \right\rangle  \leq \frac{6\mu}{L^2}\Vert \delta^c_t (e_c)^{q-\frac{1}{2}} \Vert^2 + \frac{L^2\mu}{24}\Vert \delta^c_t (Q_c)^{q-\frac{1}{2}} \Vert^2, \\
        &\frac{\mu h^2}{12} \left\langle \delta^c_t (\rho_c)^{q-\frac{1}{2}}, \delta^c_t (Q_c)^{q-\frac{1}{2}} \right\rangle \leq \frac{\mu h^2}{18} \Vert\delta^c_t (\rho_c)^{q-\frac{1}{2}}\Vert^2 + \frac{\mu h^2}{32} \Vert\delta^c_t (Q_c)^{q-\frac{1}{2}}\Vert^2. 
    \end{aligned}
\end{equation*}
Thus 
\begin{equation}\label{eq-mues1}
    \mu\left\langle \delta^c_t (\rho_c)^{q-\frac{1}{2}} , \delta^c_t (e_c)^{q-\frac{1}{2}} \right\rangle \leq \mu\left( \frac{h^2}{32} + \frac{L^2}{24} \right) \Vert \delta^c_t (Q_c)^{q-\frac{1}{2}} \Vert^2.
\end{equation}

Next, we estimate the terms containing the operator $\Psi$. Note that $(e_c)^{q-\frac{1}{2}}=(e_c)^{q-1}+ \frac{\tau_c}{2}\delta^c_t (e_c)^{q-\frac{1}{2}} $.
with the help of Lemma \ref{lem-b}, it holds that 
\begin{equation*}
        \left\langle\Psi(~ \cdot ~,  (e_c)^{q-\frac{1}{2}} ), \delta^c_t (e_c)^{q-\frac{1}{2}}\right\rangle 
        = \left\langle\Psi( ~\cdot~,  (e_c)^{q-1} ), \delta^c_t (e_c)^{q-\frac{1}{2}}\right\rangle.
\end{equation*}
Sequentially applying the definition of $\Psi$, Lemmas \ref{lem-a} and \ref{lem-e}, \eqref{eq-notan2}, and Cauchy-Schwarz inequality, we can conclude that
\begin{equation*}
    \begin{aligned}
        -&\left\langle \Psi\left(U^{q-\frac{1}{2}},U^{q-\frac{1}{2}}\right) - \Psi\left((u_c)^{q-\frac{1}{2}},(u_c)^{q-\frac{1}{2}}\right), \delta^c_t (e_c)^{q-\frac{1}{2}} \right\rangle \\
        \leq&  C_0 \left( |(e_c)^{q-\frac{1}{2}}|_{1} + \Vert  (e_c)^{q-\frac{1}{2}} \Vert \right) \Vert \delta^c_t (e_c)^{q-\frac{1}{2}} \Vert \\
        &+ \frac{1}{3}\left( 2\Vert (e_c)^{q-\frac{1}{2}} \Vert_{\infty} \cdot |(e_c)^{q-1}|_{1} + \Vert (e_c)^{q-1} \Vert_{\infty} \cdot |(e_c)^{q-\frac{1}{2}}|_{1} \right)\Vert \delta^c_t (e_c)^{q-\frac{1}{2}} \Vert, 
    \end{aligned}
\end{equation*}
and 
\begin{equation*}
    \begin{aligned}
        &\frac{ h^2}{2}\left\langle \Psi\left(W^{q-\frac{1}{2}},U^{q-\frac{1}{2}}\right) - \Psi\left((w_c)^{q-\frac{1}{2}},(u_c)^{q-\frac{1}{2}}\right), \delta^c_t (e_c)^{q-\frac{1}{2}} \right\rangle \\
        \leq &  \frac{ h^2}{6}\left( 2C_0|(e_c)^{q-\frac{1}{2}}|_1 + C_0 \Vert (e_c)^{q-\frac{1}{2}} \Vert + 2C_0 \Vert (\rho_c)^{q-\frac{1}{2}} \Vert + C_0 |(\rho_c)^{q-\frac{1}{2}}|_1\right) \Vert \delta^c_t (e_c)^{q-\frac{1}{2}} \Vert \\
        &+ \frac{ h^2}{6}\left( 2\Vert (\rho_c)^{q-\frac{1}{2}} \Vert_{\infty} \cdot |(e_c)^{q-1}|_1 + \Vert (e_c)^{q-1} \Vert_{\infty} \cdot |(\rho_c)^{q-\frac{1}{2}}|_1\right) \Vert \delta^c_t (e_c)^{q-\frac{1}{2}} \Vert, \\ 
    \end{aligned}
\end{equation*}
where the derivation details of this two inequalities can refer to [(3.49) and (3.50), \citenum{wang2021pointwise}].
Following the flexible application of Cauchy-Schwarz inequality, Young's inequality, and \eqref{eq-errall1}, we yield
\begin{equation}\label{eq-Psies1}
    \begin{aligned}
        &-\left\langle \Psi\left(U^{q-\frac{1}{2}},U^{q-\frac{1}{2}}\right) - \Psi\left((u_c)^{q-\frac{1}{2}},(u_c)^{q-\frac{1}{2}}\right), \delta^c_t (e_c)^{q-\frac{1}{2}} \right\rangle \\
        \leq & C_0  \left( |(e_c)^{q-\frac{1}{2}}|_1  + \Vert  (e_c)^{q-\frac{1}{2}} \Vert \right) \Vert \delta^c_t (e_c)^{q-\frac{1}{2}} \Vert  + \frac{ \sqrt{L}}{2} |(e_c)^{q-\frac{1}{2}}|_1 \cdot \Vert \delta^c_t (e_c)^{q-\frac{1}{2}} \Vert \\
        \leq &  \frac{1}{4} \Vert \delta^c_t (e_c)^{q-\frac{1}{2}} \Vert^2 + 2\left(  C_0 + \frac{\sqrt{L}}{2} \right)^2 |(e_c)^{q-\frac{1}{2}}|_1^2 + 2C_0^2  \Vert  (e_c)^{q-\frac{1}{2}} \Vert^2 \\
        \leq&   \frac{1}{4} \Vert \delta^c_t (e_c)^{q-\frac{1}{2}} \Vert^2 + \left( 4 C_0^2  +  L + \frac{C_0^2 L^2}{3}\right) |(e_c)^{q-\frac{1}{2}}|_1^2 .
    \end{aligned}
\end{equation}
Similarly,  we have
\begin{equation}\label{eq-Psies2}
    \begin{aligned}
        &\frac{h^2}{2}\left\langle \Psi\left(W^{q-\frac{1}{2}},U^{q-\frac{1}{2}}\right) - \Psi\left((w_c)^{q-\frac{1}{2}},(u_c)^{q-\frac{1}{2}}\right), \delta^c_t (e_c)^{q-\frac{1}{2}} \right\rangle  \\
        \leq & \frac{ h^2}{6}\left( 2C_0 + \frac{C_0 L}{\sqrt{6}} \right)|(e_c)^{q-\frac{1}{2}}|_1 \cdot \Vert \delta^c_t (e_c)^{q-\frac{1}{2}} \Vert + \frac{ h^2}{6}\left( 2C_0 + \frac{2C_0}{h} + \frac{3\sqrt{L}}{h} \right)\Vert (\rho_c)^{q-\frac{1}{2}} \Vert \cdot \Vert \delta^c_t (e_c)^{q-\frac{1}{2}} \Vert \\
        \leq & \frac{1}{4} \Vert \delta^c_t (e_c)^{q-\frac{1}{2}} \Vert^2 + \frac{ h^4}{18}\left( 2C_0 + \frac{C_0 L}{\sqrt{6}} \right)^2|(e_c)^{q-\frac{1}{2}}|_1^2 + \frac{ h^4}{18}\left(\frac{4C_0}{h} + \frac{3\sqrt{L}}{h} \right)^2 \Vert (\rho_c)^{q-\frac{1}{2}} \Vert^2 \\
        \leq & \ \frac{1}{4} \Vert \delta^c_t (e_c)^{q-\frac{1}{2}} \Vert^2 + \left( \frac{4C_0^2 }{9} + \frac{C_0^2 L^2}{54} \right) |(e_c)^{q-\frac{1}{2}}|_1^2 + h^2\left( \frac{16C_0^2 }{9} + L  \right)\Vert (\rho_c)^{q-\frac{1}{2}} \Vert^2.
    \end{aligned}
\end{equation}

Finally, by substituting \eqref{eq-alphaes1}-\eqref{eq-Psies2} into \eqref{eq-inproduct2}, and invoking Lemma \ref{lem-f}, we obtain
\begin{equation}\label{eq-productes1}
    \begin{split}
        &\Vert \delta^c_t (e_c)^{q-\frac{1}{2}} \Vert^2 + \frac{\lambda}{2\tau_c}\left[ \left( |(e_c)^q|_{1}^2 - |(e_c)^{q-1}|_{1}^2 \right) + \frac{h^2}{12}\left( \Vert (\rho_c)^q \Vert^2 - \Vert (\rho_c)^{q-1} \Vert^2\right) - \frac{h^4}{144}\left( |(\rho_c)^q|_{1}^2 - |(\rho_c)^{q-1}|_{1}^2 \right) \right] \\
        \leq & \frac{3}{4}\Vert \delta^c_t (e_c)^{q-\frac{1}{2}} \Vert^2 + \left(L + \frac{40 C_0^2 }{9} + \frac{19 C_0^2 L^2}{54} \right) |(e_c)^{q-\frac{1}{2}}|_1^2 +  h^2\left( \frac{16C_0^2  }{9} + L  \right)\Vert (\rho_c)^{q-\frac{1}{2}} \Vert^2\\
        &+ \frac{\lambda h^2}{12}\left\langle (\rho_c)^{q-\frac{1}{2}},\delta^c_t (Q_c)^{q-\frac{1}{2}} \right\rangle + \lambda\left\langle (Q_c)^{q-\frac{1}{2}}, \delta^c_t (e_c)^{{q-\frac{1}{2}}} \right\rangle + \left\langle (R_c)^{q-\frac{1}{2}},  \delta^c_t (e_c)^{q-\frac{1}{2}}\right\rangle \\
        &+ \mu\left( \frac{h^2}{32} + \frac{L^2}{24} \right) \Vert \delta^c_t (Q_c)^{q-\frac{1}{2}} \Vert
        + 2|(e_c)^{q-\frac{1}{2}}|_1^2 + \frac{ h^4}{18}|(\rho_c)^{q-\frac{1}{2}}|_1^2\\
        \leq & \Vert \delta^c_t (e_c)^{q-\frac{1}{2}} \Vert^2 + \left(L + \frac{40 C_0^2  }{9} + \frac{19 C_0^2 L^2 }{54} + 2 \right) |(e_c)^{q-\frac{1}{2}}|_1^2 + \frac{h^2 }{9} \left(16C_0^2 + 9L + 3 \right)\Vert (\rho_c)^{q-\frac{1}{2}} \Vert^2 \\
        &+ \left( \frac{\lambda^2}{64} + \frac{\mu }{32} + \frac{\mu L^2}{24} \right) \Vert \delta^c_t (Q_c)^{q-\frac{1}{2}}\Vert^2 
        + 2\lambda^2\Vert (Q_c)^{q-\frac{1}{2}}\Vert^2 + 2\Vert (R_c)^{q-\frac{1}{2}}\Vert^2, \quad 1\leq q\leq n.  \\
    \end{split}
\end{equation}

According to Lemma \ref{lem-a}, we have
\begin{equation*}
    \frac{h^2}{9}\Vert (\rho_c)^{q-\frac{1}{2}} \Vert^2 = \frac{h^2}{18} \left(\Vert (\rho_c)^{q} \Vert^2 + \Vert (\rho_c)^{q-1}\Vert^2 \right)
    \leq \frac{h^2}{12} \left(\Vert (\rho_c)^q \Vert^2 + \Vert (\rho_c)^{q-1} \Vert^2\right)- \frac{h^4}{144}\left(|(\rho_c)^q|_{1}^2 + |(\rho_c)^{q-1}|_{1}^2\right).
\end{equation*}
Thus, if we set 
\begin{equation*}
    G^q = |(e_c)^q|_{1}^2 + \frac{h^2}{12}\Vert (\rho_c)^q \Vert^2 - \frac{h^4}{144}|(\rho_c)^q|_{1}^2,\quad 0 \leq q \leq N_c,
\end{equation*}
and use the truncation errors \eqref{eq-ErrEst} for $\eta=c$,
then \eqref{eq-productes1} can be simplified as 
\begin{equation}\label{eq-simproductes1}
   \begin{aligned}
        \frac{1}{2\tau_c}(G^q - G^{q-1}) 
        \leq C_1(G^q+G^{q-1}) +C_2(\tau_c^2+h^4)^2,  \quad 1 \leq q \leq n,
   \end{aligned} 
\end{equation}
where the constants 
\begin{equation}\label{eq-twoconsta}
    \begin{aligned}
    &C_1 = \max\left\{ \frac{L}{2\lambda} + \frac{20 C_0^2}{9\lambda} + \frac{19 C_0^2 L^2 }{108\lambda} + \frac{1}{\lambda}~,~ \frac{16C_0^2}{\lambda} + \frac{9L}{\lambda} + \frac{3}{\lambda}\right\}, \\
    &C_2 = \left( \frac{\lambda}{64} + \frac{\mu }{32\lambda} + \frac{\mu L^2}{24\lambda} + 2\lambda + \frac{2}{\lambda}   \right)C^2 L.
    \end{aligned}
\end{equation}
For the index $q$ of \eqref{eq-simproductes1}, by summing up from $1$ to $n$, we yield
\begin{equation*}
    \begin{aligned}
        G^{n} - G^0 
        &\leq 2C_1\tau_c G^{n} + 4C_1\tau_c \sum_{q=0}^{n-1} G^q + 2C_2 \tau_c \left( \tau_c^2 + h^4 \right)^2.
    \end{aligned}
\end{equation*}
Here, we also need to analyze the positivity of $G^0$. According to Lemma \ref{lem-a} and \eqref{eq-rhoes0}, it holds that 
\begin{equation}\label{eq-esG0a}
       0 \leq \frac{h^2}{18}\Vert (\rho_c)^0 \Vert^2 \leq G^0 = \frac{h^2}{12}\Vert (\rho_c)^0 \Vert^2 - \frac{h^4}{144}|(\rho_c)^0|_{1}^2 \leq \Vert (\rho_c)^0 \Vert^2 \leq \frac{9}{4}L\left(Ch^4\right)^2.
\end{equation}
Moreover, when $2C_1\tau_c<1/3$, we also have 
\begin{equation*}
    G^{n} 
    \leq 6C_1\tau_c \sum_{q=0}^{n-1} G^q + \frac{3}{2}G^0 + \frac{C_2}{2C_1}\left( \tau_c^2 + h^4 \right)^2.
\end{equation*}
Applying the Grownwall inequality [Lemma, \citenum{Sloan1986time}] and Lemma \ref{lem-a}, then we obtain 
\begin{equation*}
    |(e_c)^{n}|_{1}^2 \leq G^{n} \leq  \exp(6C_1 T)\cdot\left[ \frac{3}{2}G^0 + \frac{C_2}{2C_1} \left( \tau_c^2 + h^4 \right)^2 \right] \leq C^2\left( \tau_c^2 + h^4 \right)^2,
\end{equation*}
which implies that the proof of \eqref{eq-cvg1} has been completed.
\end{proof}

\begin{corly}
    According to Lemma \ref{lem-a} and \eqref{eq-cvg1}, we can derive the $L^2$ norm and maximum norm error estimates on the coarse grid, i.e.,
        \begin{align}
            &\Vert (e_c)^q \Vert \leq \frac{L}{\sqrt{6}}|(e_c)^q|_{1} \leq \frac{L C}{\sqrt{6}} (\tau_c^2 + h^4), \quad 0\leq q \leq N_c, \label{eq-errL2} \\
            &\Vert (e_c)^q \Vert_{\infty} \leq \frac{\sqrt{L}}{2}|(e_c)^q|_{1} \leq \frac{C\sqrt{L}}{2} (\tau_c^2 + h^4),\quad 0\leq q \leq N_c. \label{eq-errmax} 
        \end{align}
\end{corly}

\subsection{Error estimate for Step \Rmnum{2}}
Before stating the error estimates for the interpolation formulas \eqref{eq-timeintepo}-\eqref{eq-stwintepo}, we first introduce the following notation:
\begin{equation*}
    (e_f)_p^k =  U_p^k - (u_f)_p^k, \quad (\rho_f)_p^k = W_p^k-(w_f)_p^k , \quad 1\leq p\leq M, \ 0\leq k \leq N_f.
\end{equation*}

\begin{thm}
    Suppose $\{u(x,t),w(x,t) \mid (x,t)\in \Omega\times [0,T]  \}$ is the exact solution of problem \eqref{eq-bbmbnew}-\eqref{eq-twordfm} and $\{(u_f)_p^k, (w_f)_p^k \mid 1\leq p\leq M, ~ 0\leq k \leq N_f \}$ is the rough solution on the fine grid obtained by the interpolation formulas \eqref{eq-timeintepo}-\eqref{eq-stwintepo}.
    If $u(x,t) \in C^{4,2}(\bar{\Omega}\times [0,T])$, then
    \begin{equation}\label{eq-cvg2}
        \|(e_f)^k\| \leq C (\tau_c^2 + h^4), \quad \|(e_f)^k\|_{\infty} \leq \widetilde{C}, \quad  0\leq k \leq N_f,
    \end{equation}
    where the positive constants $C$, $\widetilde{C} $ are independent of $\tau_c$ and $h$.
\end{thm}

\begin{proof}
Assume that $u(x,t) \in C^{4,2}(\bar{\Omega}\times [0,T])$. For the fixed spatial grid node $x_{p}$, $1\leq p\leq M$, considering the truncation error of linear Lagrange interpolation in temporal direction, we easily yield
\begin{equation}\label{eq-tmintp}
    \begin{aligned}
    U_{p}^{(q-1)\beta_{\tau}+ r} 
    = &\left(1-\frac{r}{\beta_{\tau}}\right)U_{p}^{(q-1)\beta_{\tau}} + \frac{r}{\beta_{\tau}}U_{p}^{q\beta_{\tau}} + \frac{u_{tt}\left(x_{p},\xi\right)}{2}\left[(t_f)_{(q-1)\beta_{\tau}+ r}-(t_c)_{q-1}\right] \\
    &\cdot \left[(t_f)_{(q-1)\beta_{\tau}+ r}-(t_c)_{q}\right],
    \quad  \xi\in\left((t_c)_{(q-1)},(t_c)_{q}\right). 
    \end{aligned}
\end{equation}
Subtracting \eqref{eq-timeintepo} from \eqref{eq-tmintp}, then we have
\begin{equation*}
    \begin{aligned}
    (e_f)_{p}^{(q-1)\beta_{\tau}+ r} 
    =& \left(1-\frac{r}{\beta_{\tau}}\right)(e_c)_{p}^{q-1} + \frac{r}{\beta_{\tau}}(e_c)_{p}^{q} \\
    &+\frac{u_{tt}\left(x_{p},\xi\right)}{2}\left[(t_f)_{(q-1)\beta_{\tau}+ r}-(t_c)_{q-1}\right]\left[(t_f)_{(q-1)\beta_{\tau}+ r}-(t_c)_{q}\right].
    \end{aligned}
\end{equation*} 
Applying the triangle inequality, it holds that
    \begin{align*}
    &\|(e_f)^{(q-1)\beta_{\tau}+ r}\| \leq \left(1-\frac{r}{\beta_{\tau}}\right)\|(e_c)^{q-1}\| + \frac{r}{\beta_{\tau}}\|(e_c)^{q}\| + \frac{CL}{2}\tau_c^2,  \\
    &|(e_f)^{(q-1)\beta_{\tau}+ r}|_{1} 
    \leq \left(1-\frac{r}{\beta_{\tau}}\right)|(e_c)^{q-1}|_{1} + \frac{r}{\beta_{\tau}}|(e_c)^{q}|_{1} + \frac{CL}{2}\tau_c^2 .
    \end{align*}
After combining the above estimates  with \eqref{eq-errL2} and \eqref{eq-cvg1}, we obtain
\begin{align}
    &\|(e_f)^{k}\| \leq \frac{LC}{\sqrt{6}} (\tau_c^2 + h^4) +\frac{CL}{2}\tau_c^2 \leq C(\tau_c^2 + h^4), \ 0\leq k\leq N_f, \label{eq-tmintperrL2} \\
    &\|(e_f)^{k}\|_{\infty} \leq \frac{\sqrt{L}}{2}|(e_f)^{k}|_{1} \leq \frac{C \sqrt{L}}{2} (\tau_c^2 + h^4) +  \frac{C\sqrt{L^3}}{4}\tau_c^2  \leq \widetilde{C}, \quad 0\leq k\leq N_f. \label{eq-tmintperrmax} 
\end{align}
where $\widetilde{C}>0$ is independent of $\tau_c$ and $h$.
This completes the proof.
\end{proof}

Now, we take $\eta=f$ in \eqref{eq-pbm3b}, and subtract \eqref{eq-stwintepo} from \eqref{eq-pbm3b}, then we yield
\begin{equation}\label{eq-tworelerr}
    (\rho_f)_p^{k} = \delta_{xx} (e_f)_p^k - \frac{h^2}{12}\delta_{xx} (\rho_f)_p^k + (Q_{f})_p^k, \quad 1\leq p\leq M, \ 0\leq k \leq N_{f},
\end{equation}
and it can be rewritten as the following vector form,
\begin{equation*}
    (\rho_f)^{k} = \left(\mathcal{I}+\frac{h^2}{12}\delta_{xx}\right)^{-1}\delta_{xx}(e_f)^{k} + \left(I+\frac{h^2}{12}\delta_{xx}\right)^{-1}(Q_f)^k,
\end{equation*}
where $\mathcal{I}$ is the identity operator. Denote $A=\mathcal{I}+\frac{h^2}{12}\delta_{xx}$. It is easy to check that $A$ is a circulant matrix, then its eigenvalues $(\sigma_{A})_p = 1-\frac{1}{3}\sin^2\left(\frac{p\pi}{M}\right)$ can be calculated by discrete Fourier transform, see \cite{gray2006toeplitz}. Thus, the eigenvalues of $A^{-1}$ satisfy $1\leq (\sigma_{A^{-1}})_p \leq \frac{3}{2}$. Based on this conclusion and Lemma \ref{lem-a}, we have  
\begin{equation}\label{eq-rhoL2st2}
    \begin{aligned}
        \| (\rho_f)^{k}  \| 
        &\leq \| A^{-1}  \left(\delta_{xx}(e_f)^{k} + (Q_f)^k\right)\| \leq \| A^{-1}\|_{spec}\left(\|\delta_{xx}(e_f)^{k}\| + \|(Q_f)^k\| \right) \\
        &\leq \frac{3}{2} \left(\frac{4}{h^2}\|(e_f)^{k}\| + \|(Q_f)^k\|\right) = \frac{6}{h^2}\|(e_f)^{k}\| + \frac{3}{2}\|(Q_f)^k\|.
    \end{aligned}
\end{equation}
Finally, by using the definition of the operator $\delta_{xx}$, \eqref{eq-tworelerr} can be rewritten as 
\begin{equation*}
    \frac{5}{6}(\rho_f)^{k}_p = \frac{1}{h^2}\left( (e_f)^k_{p+1} -2(e_f)^{k}_p + (e_f)^k_{p-1} \right)  - \frac{1}{12}\left( (\rho_f)^{k}_{p-1}  + (\rho_f)^{k}_{p+1}\right) + (Q_f)_p^k.
\end{equation*}
Therefore, under the periodic boundary condition, we can derive
\begin{equation*}
    \frac{5}{6}|(\rho_f)^{k}_p| \leq \frac{4}{h^2} \|(e_f)^k\|_{\infty} + \frac{1}{6}\|(\rho_f)^k\|_{\infty} + |(Q_f)_p^k|, \quad 1\leq p\leq M.
\end{equation*} 
we know that $|(Q_f)_p^k| \leq \widehat{C}h^4$, $\|(e_f)^k\|_{\infty}\leq \widetilde{C}$ from \eqref{eq-ErrEst} and \eqref{eq-cvg2}, respectively.
Since the above inequality holds for all $p$, there exists a constant $\widetilde{C}>0$ such that
\begin{equation}\label{eq-rhomaxst2}
    \left\|(\rho_f)^k\right\|_{\infty}\leq \frac{6}{h^2}\left(\left\|(e_f)^k\right\|_{\infty} +  \frac{3}{2}\widehat{C}h^6 \right)
    \leq \frac{6}{h^2} \widetilde{C}.
\end{equation}

\subsection{Error estimate for Step \Rmnum{3}}
Finally, we analyze the convergence of the scheme \eqref{eq-pbm6a}-\eqref{eq-pbm6d}.
Denote 
\begin{equation*}
    e_p^k =  U_p^k - u_p^k, \quad \rho_p^k = W_p^k-w_p^k , \quad 1\leq p\leq M, \quad 0\leq k \leq N_f.
\end{equation*}

Take $\eta=f$ in scheme \eqref{eq-pbm3a}-\eqref{eq-pbm3b}. Subtracting the scheme \eqref{eq-pbm6a}-\eqref{eq-pbm6d} from \eqref{eq-pbm3a}-\eqref{eq-pbm3b}, we derive the following error system 
\begin{subequations}\label{eq-errpbm1}
\begin{numcases}{}
    \delta^{f}_t e_p^{k-\frac{1}{2}} - \mu\delta^{f}_t \rho_p^{k-\frac{1}{2}}  + \Psi\left(U_p^{k-\frac{1}{2}},U_p^{k-\frac{1}{2}}\right) - \Psi\left((u_f)_p^{k-\frac{1}{2}},u_p^{k-\frac{1}{2}}\right) \notag \\
    \quad - \frac{h^2}{2} \left[\Psi\left(W_p^{k-\frac{1}{2}},U_p^{k-\frac{1}{2}}\right) - \Psi\left((w_f)_p^{k-\frac{1}{2}},u_p^{k-\frac{1}{2}}\right)\right] + \Delta_x e_p^{k-\frac{1}{2}} - \frac{h^2}{6}\Delta_x \rho_p^{k-\frac{1}{2}}  \notag \\
    \quad - \lambda \rho_p^{k-\frac{1}{2}}  = (R_f)_p^{k-\frac{1}{2}}, \quad 1\leq p\leq M, \quad 0\leq k\leq N_{f},  \label{eq-errjbm2a}\\
    \rho_p^k = \delta_{xx} e_p^k - \frac{h^2}{12}\delta_{xx} \rho_p^k + (Q_f)_p^k,\quad 1\leq p\leq M, \quad 0\leq k\leq N_{f}, \label{eq-errjbm2b} \\
    e_p^k = e_{p+M}^k, \quad \rho_p^k = \rho_{p+M}^k, \quad 1\leq p\leq M, \quad 0\leq k\leq N_{f}. \label{eq-errjbm2c} 
\end{numcases}
\end{subequations}

\begin{thm}
    Suppose $\{u(x,t),w(x,t) \mid (x,t)\in \Omega\times [0,T]  \}$ is the exact solution of problem \eqref{eq-bbmbnew}-\eqref{eq-twordfm} and $\{u_p^k, w_p^k \mid 1\leq p\leq M, ~ 0\leq k \leq N_f \}$ is the corrected solution on the fine grid evaluated by the problem \eqref{eq-pbm6a}-\eqref{eq-pbm6d}.
    If $2\tau_f C_3 < 1/3$, then
    \begin{equation}\label{eq-cvg3}
        |e^k|_{1} \leq C \left(\tau_c^2 + \tau_f^2 + h^4\right), \quad  0\leq k \leq N_f.
    \end{equation}
\end{thm}

\begin{proof}
    Due to the initial value is known, so \eqref{eq-cvg3} clearly holds when $k=0$. Similar to \eqref{eq-rhoes0}, we also get
    \begin{equation}\label{eq-rhoes00}
       \| \rho^0 \|^2 \leq \frac{9}{4}L\left(Ch^4\right)^2.
    \end{equation}
    Now, we prove that \eqref{eq-cvg3} holds for $k=n$. Taking an inner product of \eqref{eq-errjbm2a} with $\delta^f_t e^{k-\frac{1}{2}}$, we have
    \begin{equation}\label{eq-inproduct3}
        \begin{aligned}
            &\Vert \delta^f_t e^{k-\frac{1}{2}} \Vert^2 - \mu \left\langle \delta^f_t\rho^{k-\frac{1}{2}} , \delta^f_t e^{k-\frac{1}{2}} \right\rangle 
            +  \left\langle\Psi\left(U^{k-\frac{1}{2}},U^{k-\frac{1}{2}}\right) - \Psi\left((u_f)^{k-\frac{1}{2}},u^{k-\frac{1}{2}}\right), \delta^f_t e^{k-\frac{1}{2}} \right\rangle \\
            &- \frac{ h^2}{2} \left\langle \Psi\left(W^{k-\frac{1}{2}},U^{k-\frac{1}{2}}\right) - \Psi\left((w_f)^{k-\frac{1}{2}},u^{k-\frac{1}{2}}\right), \delta^f_t e^{k-\frac{1}{2}}\right\rangle - \lambda \left\langle \rho^{k-\frac{1}{2}},\delta^f_t e^{k-\frac{1}{2}} \right\rangle \\
            &+\left\langle \Delta_x e^{k-\frac{1}{2}}, \delta^f_t e^{k-\frac{1}{2}} \right\rangle -\frac{h^2}{6} \left\langle \Delta_x \rho^{k-\frac{1}{2}}, \delta^f_t e^{k-\frac{1}{2}}\right\rangle = \left\langle (R_f)^{k-\frac{1}{2}},  \delta^f_t e^{k-\frac{1}{2}}\right\rangle, \quad 1\leq k\leq n.
        \end{aligned}
    \end{equation}
    Since some inner product terms in \eqref{eq-inproduct3} are consistent with \eqref{eq-inproduct2}, some details will be omitted. 
    For the second term, we employ \eqref{eq-lemd2} in Lemma \ref{lem-d} to get
    \begin{equation*}
        \begin{aligned}
        \mu \left\langle \delta^f_t\rho^{k-\frac{1}{2}} , \delta^f_t e^{k-\frac{1}{2}} \right\rangle 
        \leq& - \mu|\delta^f_t e^{k-\frac{1}{2}} |_1^2 - \frac{\mu h^2}{18}\| \delta^f_t\rho^{k-\frac{1}{2}}  \|^2 \\
        &+ \frac{\mu h^2}{12}\left\langle \delta^f_t\rho^{k-\frac{1}{2}}, \delta_t^f(Q_f)^{k-\frac{1}{2}} \right\rangle + \mu\left\langle  \delta_t^f(Q_f)^{k-\frac{1}{2}}, \delta^f_t e^{k-\frac{1}{2}} \right\rangle.
        \end{aligned}
    \end{equation*}
    Using the Cauchy-Schwarz inequality, Young's inequality and Lemma \ref{lem-a}, we have
    \begin{equation*}
        \begin{aligned}
            &\frac{\mu h_f^2}{12}\left\langle \delta^f_t\rho^{k-\frac{1}{2}}, \delta_t^f(Q_f)^{k-\frac{1}{2}} \right\rangle \leq \frac{\mu h^2}{18} \| \delta^f_t \rho^{k-\frac{1}{2}} \|^2 + \frac{\mu h^2}{32} \| \delta_t^f (Q_f)^{k-\frac{1}{2}} \|^2, \\
            &\mu\left\langle  \delta_t^f(Q_f)^{k-\frac{1}{2}}, \delta^f_t e^{k-\frac{1}{2}} \right\rangle \leq \frac{\mu L}{\sqrt{6}}\|\delta_t^f(Q_f)^{k-\frac{1}{2}}\| \cdot |\delta^f_t e^{k-\frac{1}{2}}|_1 \leq \frac{\mu}{2} |\delta^f_t e^{k-\frac{1}{2}}|_1^2 + \frac{\mu L^2}{12} \| \delta_t^f(Q_f)^{k-\frac{1}{2}} \|^2.
        \end{aligned}
    \end{equation*}
    This suggests that
    \begin{equation}\label{eq-mues2}
        \mu \left\langle \delta^f_t\rho^{k-\frac{1}{2}} , \delta^f_t e^{k-\frac{1}{2}} \right\rangle \leq -\frac{\mu}{2} |\delta^f_t e^{k-\frac{1}{2}}|_1^2 + \mu\left(\frac{ h_f^2}{32} + \frac{L^2}{12} \right) \| \delta_t^f (Q_f)^{k-\frac{1}{2}} \|^2 .
    \end{equation}    
    Regarding the terms containing operator $\Delta^f_x$, we directly use \eqref{eq-alphaes1}, that is,
    \begin{equation}\label{eq-Deltaes1}
        \begin{aligned}
            -\left\langle \Delta_x e^{k-\frac{1}{2}}, \delta^f_t e^{k-\frac{1}{2}} \right\rangle 
            &\leq \frac{1}{8}\Vert \delta^f_t e^{k-\frac{1}{2}} \Vert^2 + 2|e^{k-\frac{1}{2}}|_{1}^2, \\
            \frac{h^2}{6}\left\langle \Delta_x \rho^{k-\frac{1}{2}}, \delta^f_t e^{k-\frac{1}{2}} \right\rangle &\leq \frac{1}{8}\Vert \delta^f_t e^{k-\frac{1}{2}} \Vert^2 + \frac{h^4}{18}|\rho^{k-\frac{1}{2}}|_{1}^2.
        \end{aligned}
    \end{equation}
    Afterwards, we focus on the estimates of the inner product term containing the operator $\Psi$. 
    Recall the notation $(e_f)_p^k =  U_p^k - (u_f)_p^k$. 
    By combining the definition of $\Psi$ with Lemmas \ref{lem-b} and \ref{lem-e}, we get 
    \begin{equation*}
        \begin{aligned}
            &-\left\langle\Psi\left(U^{k-\frac{1}{2}},U^{k-\frac{1}{2}}\right) - \Psi\left((u_f)^{k-\frac{1}{2}},u^{k-\frac{1}{2}}\right), \delta^f_t e^{k-\frac{1}{2}} \right\rangle \\ 
            = &- \left\langle\Psi\left(U^{k-\frac{1}{2}},e^{k-\frac{1}{2}}\right) + \Psi\left((e_f)^{k-\frac{1}{2}},U^{k-\frac{1}{2}}\right) - \Psi\left((e_f)^{k-\frac{1}{2}},e^{k-\frac{1}{2}}\right), \delta^f_t e^{k-\frac{1}{2}} \right\rangle \\
            = &-\frac{1}{3}\left\langle 2U^{k-\frac{1}{2}} \Delta_x e^{k-\frac{1}{2}} +   \frac{1}{2}D_{+} e^{k-\frac{1}{2}} \; \delta_x^{+\frac{1}{2}} U^{k-\frac{1}{2}} + \frac{1}{2}D_{-} e^{k-\frac{1}{2}} \; \delta_x^{-\frac{1}{2}} U^{k-\frac{1}{2}} , \delta_t^f e^{k-\frac{1}{2}} \right\rangle \\
            &- \frac{1}{3}\left\langle (e_f)^{k-\frac{1}{2}} \Delta_x U^{k-\frac{1}{2}} , \delta^f_t e^{k-\frac{1}{2}} \right\rangle + \frac{1}{3}\left\langle (e_f)^{k-\frac{1}{2}} U^{k-1}  , \Delta_x\left(\delta^f_t e^{k-\frac{1}{2}}\right) \right\rangle\\ 
            &+ \frac{1}{3}\left\langle (e_f)^{k-\frac{1}{2}} \Delta_x e^{k-\frac{1}{2}} , \delta^f_t e^{k-\frac{1}{2}} \right\rangle - \frac{1}{3}\left\langle  (e_f)^{k-\frac{1}{2}} e^{k-\frac{1}{2}}  , \Delta_x\left(\delta^f_t e^{k-\frac{1}{2}}\right) \right\rangle.
        \end{aligned}
    \end{equation*}
    Furthermore, by flexibly applying the Cauchy-Schwarz inequality, Young's inequality, \eqref{eq-notan2}, \eqref{eq-cvg2} and Lemma \ref{lem-a}, we derive
    \begin{equation}\label{eq-Psies3}
        \begin{split}
            &-\left\langle\Psi\left(U^{k-\frac{1}{2}},U^{k-\frac{1}{2}}\right) - \Psi\left((u_f)^{k-\frac{1}{2}},u^{k-\frac{1}{2}}\right), \delta^f_t e^{k-\frac{1}{2}} \right\rangle  \\
            \leq& \left(\frac{2C_0}{3}|e^{k-\frac{1}{2}} |_1 + \frac{C_0}{3}\| e^{k-\frac{1}{2}}\|\right)\|\delta^f_t e^{k-\frac{1}{2}}\| 
            + \frac{C_0}{3}\| (e_f)^{k-\frac{1}{2}}\| \left(\|\delta^f_t e^{k-\frac{1}{2}}\| + |\delta^f_t e^{k-\frac{1}{2}}|_1 \right)  \\
            &+ \frac{1}{3}\| (e_f)^{k-\frac{1}{2}}\|_{\infty} \left( |e^{k-\frac{1}{2}}|_{1} \cdot \| \delta^f_t e^{k-\frac{1}{2}}\| + \|e^{k-\frac{1}{2}}\| \cdot | \delta^f_t e^{k-\frac{1}{2}}|_{1} \right)  \\
            \leq& \left(\frac{2C_0}{3} + \frac{C_0 L}{3\sqrt{6}}\right)|e^{k-\frac{1}{2}} |_1 \cdot \|\delta^f_t e^{k-\frac{1}{2}}\| + \frac{C_0}{3}\| (e_f)^{k-\frac{1}{2}}\| \left(\|\delta^f_t e^{k-\frac{1}{2}}\| + |\delta^f_t e^{k-\frac{1}{2}}|_1 \right)  \\
            &+ \frac{\widetilde{C}}{3}  |e^{k-\frac{1}{2}}|_{1} \cdot \| \delta^f_t e^{k-\frac{1}{2}}\| + + \frac{\widetilde{C}L}{3\sqrt{6}}|e^{k-\frac{1}{2}}|_{1} \cdot | \delta^f_t e^{k-\frac{1}{2}}|_{1} \\
            \leq& \frac{1}{4}\|\delta^f_t e^{k-\frac{1}{2}}\|^2 + \left(\frac{16C_0^2}{9} + \frac{2C_0^2L^2}{27} + \frac{4\widetilde{C}^2}{9} + \frac{\widetilde{C}^2 L^2}{27\mu}\right)|e^{k-\frac{1}{2}}|_{1}^2 
            + \frac{\mu}{4}|\delta^f_t e^{k-\frac{1}{2}}|_{1}^2 + \left(\frac{4C_0^2}{9} + \frac{2C_0^2}{9\mu} \right)\| (e_f)^{k-\frac{1}{2}}\|^2. 
        \end{split}    
    \end{equation}
    Similarly, we have
    \begin{equation*}
        \begin{aligned}
            &\frac{ h^2}{2}\left\langle\Psi\left(W^{k-\frac{1}{2}},U^{k-\frac{1}{2}}\right) - \Psi\left((w_f)^{k-\frac{1}{2}},u^{k-\frac{1}{2}}\right), \delta^f_t e^{k-\frac{1}{2}} \right\rangle \\ 
            = &\frac{ h^2}{2} \left\langle\Psi\left(W^{k-\frac{1}{2}},e^{k-\frac{1}{2}}\right) + \Psi\left((\rho_f)^{k-\frac{1}{2}},U^{k-\frac{1}{2}}\right) - \Psi\left((\rho_f)^{k-\frac{1}{2}},e^{k-\frac{1}{2}}\right), \delta^f_t e^{k-\frac{1}{2}} \right\rangle \\
            = &\frac{h^2}{6}\left\langle 2W^{k-\frac{1}{2}} \Delta_x e^{k-\frac{1}{2}} +   \frac{1}{2}D_{+} e^{k-\frac{1}{2}} \; \delta_x^{+\frac{1}{2}} W^{k-\frac{1}{2}} + \frac{1}{2}D_{-} e^{k-\frac{1}{2}} \; \delta_x^{-\frac{1}{2}} W^{k-\frac{1}{2}} , \delta^f_t e^{k-\frac{1}{2}} \right\rangle \\
            &+\frac{h^2}{6} \left\langle (\rho_f)^{k-\frac{1}{2}} \Delta_x U^{k-\frac{1}{2}} , \delta^f_t e^{k-\frac{1}{2}} \right\rangle - \frac{h^2}{6}\left\langle (\rho_f)^{k-\frac{1}{2}} U^{k-\frac{1}{2}}  , \Delta_x\left(\delta^f_t e^{k-\frac{1}{2}}\right) \right\rangle\\ 
            &- \frac{h^2}{6}\left\langle (\rho_f)^{k-\frac{1}{2}} \Delta_x e^{k-\frac{1}{2}} , \delta^f_t e^{k-\frac{1}{2}} \right\rangle + \frac{h^2}{6}\left\langle  (\rho_f)^{k-\frac{1}{2}} e^{k-\frac{1}{2}}  , \Delta_x\left(\delta^f_t e^{k-\frac{1}{2}}\right) \right\rangle,
        \end{aligned}
    \end{equation*}
    and using Cauchy-Schwarz inequality, Young's inequality, \eqref{eq-notan2}, \eqref{eq-rhomaxst2}, Lemma \ref{lem-a} again, it follows that
    \begin{equation}\label{eq-Psies4}
        \begin{split}
            &\frac{ h^2}{2}\left\langle\Psi\left(W^{k-\frac{1}{2}},U^{k-\frac{1}{2}}\right) - \Psi\left((w_f)^{k-\frac{1}{2}},u^{k-\frac{1}{2}}\right), \delta^f_t e^{k-\frac{1}{2}} \right\rangle \\
            \leq& \left(\frac{C_0 h^2}{3}|e^{k-\frac{1}{2}} |_1 + \frac{C_0 h^2}{6}\| e^{k-\frac{1}{2}}\|\right)\|\delta^f_t e^{k-\frac{1}{2}}\|  + \frac{C_0 h^2}{6}\| (\rho_f)^{k-\frac{1}{2}}\| \left( \|\delta^f_t e^{k-\frac{1}{2}}\| + |\delta^f_t e^{k-\frac{1}{2}}|_1 \right) \\
            &+ \frac{h^2}{6}\| (\rho_f)^{k-\frac{1}{2}}\|_{\infty} \left( |e^{k-\frac{1}{2}}|_{1} \cdot \| \delta^f_t e^{k-\frac{1}{2}}\| +  \|e^{k-\frac{1}{2}}\| \cdot | \delta^f_t e^{k-\frac{1}{2}}|_{1}\right) \\
            \leq& \left(\frac{C_0 h^2}{3} + \frac{C_0 L h^2}{6\sqrt{6}}\right)|e^{k-\frac{1}{2}} |_1 \cdot \|\delta^f_t e^{k-\frac{1}{2}}\| + \frac{C_0 h^2}{6}\| (\rho_f)^{k-\frac{1}{2}}\| \left( \|\delta^f_t e^{k-\frac{1}{2}}\| + |\delta^f_t e^{k-\frac{1}{2}}|_1 \right) \\
            &+ \widetilde{C}  |e^{k-\frac{1}{2}}|_{1} \cdot \| \delta^f_t e^{k-\frac{1}{2}}\| + \frac{\widetilde{C}L }{\sqrt{6}}|e^{k-\frac{1}{2}}|_{1} \cdot | \delta^f_t e^{k-\frac{1}{2}}|_{1} \\
            \leq& \frac{1}{4}\|\delta^f_t e^{k-\frac{1}{2}}\|^2 + \left(\frac{4C_0^2}{9} + \frac{C_0^2L^2}{54} + 4\widetilde{C}^2 + \frac{\widetilde{C}^2 L^2}{3\mu}  \right)|e^{k-\frac{1}{2}}|_{1}^2  
            +\frac{\mu}{4}|\delta^f_t e^{k-\frac{1}{2}}|_{1}^2 + \left( \frac{C_0^2 }{18\mu} + \frac{C_0^2 }{9} \right)h^4 \| (\rho_f)^{k-\frac{1}{2}}\|^2.
        \end{split}    
    \end{equation}

    Finally, by applying Lemmas \ref{lem-a}, \ref{lem-f}, and substituting \eqref{eq-mues2}-\eqref{eq-Psies3} into \eqref{eq-inproduct3}, it holds that
    \begin{equation}\label{eq-lastes1}
        \begin{aligned}
            &\Vert \delta^f_t e^{k-\frac{1}{2}} \Vert^2 + \frac{\lambda}{2\tau_f}\left[ \left(|e^{k}|_1^2 - |e^{k-1}|_1^2\right) + \frac{h^2}{12}\left(\|\rho^{k}\|^2 - \|\rho^{k-1}\|^2\right) - \frac{h^4}{144}\left( |\rho^{k}|_1^2 - |\rho^{k-1}|_1^2 \right)\right] \\
            \leq &\frac{3}{4}\Vert \delta^f_t e^{k-\frac{1}{2}} \Vert^2 + \left( \frac{20C_0^2}{9} + \frac{5C_0^2L^2}{54} + \frac{40\widetilde{C}^2}{9} + \frac{10\widetilde{C}^2 L^2}{27\mu} + 2\right) |e^{k-\frac{1}{2}}|_1^2 + \frac{2h^2}{9} \|\rho^{k-\frac{1}{2}}\|   \\
            &+ \left(\frac{4C_0^2}{9} + \frac{2C_0^2}{9\mu} \right)\| (e_f)^{k-\frac{1}{2}}\|^2 
            +\left( \frac{C_0^2}{18\mu} + \frac{C_0^2}{9} \right)h^4 \| (\rho_f)^{k-\frac{1}{2}}\|^2 + \mu\left(\frac{ h^2}{32} + \frac{L^2}{12} \right) \| \delta_t^f (Q_f)^{k-\frac{1}{2}} \|^2  \\
            &+ \frac{\lambda h^2}{12}\left\langle \rho^{k-\frac{1}{2}},\delta_t^f (Q_f)^{k-\frac{1}{2}} \right\rangle + \lambda\left\langle (Q_f)^{k-\frac{1}{2}},\delta_t^f e^{k-\frac{1}{2}} \right\rangle + \left\langle (R_f)^{k-\frac{1}{2}},  \delta^f_t e^{k-\frac{1}{2}}\right\rangle .
        \end{aligned}
    \end{equation}
    we now apply Cauchy-Schwarz inequality and Young's inequality to the remaining inner product terms in order to eliminate $\Vert \delta^f_t e^{k-\frac{1}{2}} \Vert^2$. Then, substituting \eqref{eq-rhoL2st2}, we yield
    \begin{equation*}
        \begin{split}
            &\frac{\lambda}{2\tau_f}\left[ \left(|e^{k}|_1^2 - |e^{k-1}|_1^2\right) + \frac{h^2}{12}\left(\|\rho^{k}\|^2 - \|\rho^{k-1}\|^2\right) - \frac{h^4}{144}\left( |\rho^{k}|_1^2 - |\rho^{k-1}|_1^2 \right)\right] \\
            \leq &\left( \frac{20C_0^2}{9} + \frac{5C_0^2L^2}{54} + \frac{40\widetilde{C}^2}{9} + \frac{10\widetilde{C}^2 L^2}{27\mu} + 2\right) |e^{k-\frac{1}{2}}|_1^2 + \frac{3h^2}{9}\|\rho^{k-\frac{1}{2}}\| + \left(\frac{76C_0^2}{9} + \frac{38C_0^2}{9\mu} \right)\| (e_f)^{k-\frac{1}{2}}\|^2 \\
            &+ \mu\left(\frac{\lambda^2 h^2}{64} + \frac{ h^2}{32} + \frac{L^2}{12} \right) \| \delta_t^f (Q_f)^{k-\frac{1}{2}} \|^2 + \left(\frac{C_0^2 h^4}{4\mu} + \frac{C_0^2 h^4}{2} + 2 \lambda^2\right)\|  (Q_f)^{k-\frac{1}{2}} \|^2 + 2\|  (R_f)^{k-\frac{1}{2}} \|^2 .
        \end{split}
    \end{equation*}
    Set
    $
        G^k = |e^k|_1^2 + \frac{h^2}{12}\| \rho^k\|^2 - \frac{h^4}{144}|\rho^k|_1^2.
    $
    Then, by taking $\eta=f$ in \eqref{eq-ErrEst}, and using \eqref{eq-cvg2}, we derive
    \begin{equation}\label{eq-simproductes2}
        \frac{1}{2\tau_f}\left(G^k - G^{k-1}\right) \leq C_3\left(G^k + G^{k-1}\right) + C_4\left( \tau_c^2 + h^4 \right)^2 + C_5\left( \tau_f^2 + h^4 \right)^2,\quad  1\leq k\leq n.
    \end{equation}
    where the constants are denoted as 
    \begin{equation*}
        \begin{aligned}
            &C_3 = \max\left\{ \frac{10C_0^2}{9\lambda} + \frac{5C_0^2L^2}{108\lambda} + \frac{20\widetilde{C}^2}{9\lambda} + \frac{5\widetilde{C}^2 L^2}{27\mu\lambda} + \frac{1}{\lambda} ~,~ \frac{3}{\lambda} \right\}, \quad
            C_4 = \left(\frac{76 C_0^2}{9\lambda} + \frac{38C_0^2}{9\mu\lambda}\right)C^2, \\
            &C_5 = \left(\frac{\mu\lambda}{64} +\frac{\mu}{32\lambda} + \frac{\mu L^2}{12\lambda}  + \frac{C_0^2 }{4\mu\lambda} + \frac{C_0^2}{2\lambda} + 2\lambda + \frac{2}{\lambda}  \right)C^2L.
        \end{aligned}
    \end{equation*}
    For the inequality \eqref{eq-simproductes2}, summing up over the index $k$ from $1$ to $n$, we get
    \begin{equation*}
        G^n - G^{0} \leq 2\tau_f C_3 G^n + 4\tau_f C_3 \sum_{k=0}^{n-1}G^k + 2\tau_fC_4\left( \tau_c^2 + h^4 \right)^2 + 2\tau_fC_5\left( \tau_f^2 + h^4 \right)^2.
    \end{equation*} 
    Obviously, $0\leq G^{0} \leq  \frac{9}{4}L(Ch^4)^2$ can be derived from \eqref{eq-rhoes00} and Lemma \ref{lem-a}. Thus, when $2\tau_f C_3 < 1/3$, we have
    \begin{equation*}
         G^n  \leq 6\tau_f C_3 \sum_{k=0}^{n-1}G^k + \frac{3}{2}G^0 + \frac{C_4}{2C_3}\left( \tau_c^2 + h^4 \right)^2 + \frac{C_5}{2C_3}\left( \tau_f^2 + h^4 \right)^2.
    \end{equation*}
    Using the Grownwall inequality and Lemma \ref{lem-a}, it holds that
    \begin{equation*}
        |e^n|_1^2 \leq G^n \leq \exp\left( 6 C_3 T \right)\left[\frac{3}{2}G^0 + \frac{C_4}{2C_3}\left( \tau_c^2 + h^4 \right)^2 + \frac{C_5}{2C_3}\left( \tau_f^2 + h^4 \right)^2 \right] \leq C^2 \left(\tau_c^2 + \tau_f^2 + h^4 \right)^2, 
    \end{equation*}
    which implies the proof of \eqref{eq-cvg3} has been completed.
\end{proof}

\begin{rmk}
Some scholars have employed Newton linearization (i.e., Taylor expansion) to design Step \Rmnum{3}, which inevitably introduces truncation errors that must be accounted for, as discussed in \cite{qiu2020time}.
As a result, the right-hand side of the corresponding inequality would include terms like $\|(e_f)^{k-\frac{1}{2}}\|^4$, ultimately leading to an error estimate on the order of $\mathcal{O}(\tau_f^2 + \tau_c^4 + h^4)$.

In contrast, the linearized scheme \eqref{eq-pbm6a}-\eqref{eq-pbm6d} introduced in Step \Rmnum{3} of our work does not rely on Taylor expansion, thereby avoiding the truncation errors associated with such linearization. Given that $\tau_c = \beta_{\tau} \tau_f$ and $\beta_{\tau} \geq 1$, our error estimate $\mathcal{O}(\tau_f^2 + \tau_c^2 + h^4)$ presented in \eqref{eq-cvg3} remains both valid and computationally feasible.

\end{rmk}

\begin{corly}
    The $L^2$ norm and maximum norm error estimates on the fine grid as follows:
        \begin{align}
            &\Vert e^k \Vert \leq \frac{L}{\sqrt{6}}|e^k|_{1} \leq \frac{L C}{\sqrt{6}} \left(\tau_c^2 + \tau_f^2 + h^4\right), \quad 0\leq k \leq N_f, \label{eq-errL2b} \\
            &\Vert e^k \Vert_{\infty} \leq \frac{\sqrt{L}}{2}|e^k|_{1} \leq \frac{C\sqrt{L}}{2} \left(\tau_c^2 + \tau_f^2 + h^4\right),\quad 0\leq k \leq N_f. \label{eq-errmaxb}  
        \end{align}
\end{corly}

\subsection{Stability}

In what follows, we analyze the stability of Step \Rmnum{1} and Step \Rmnum{3} in TTCD scheme. 
In Step \Rmnum{1}, by \eqref{eq-cvg1} and Lemma \ref{lem-a}, the boundedness of numerical solution $(u_c)_p^q$ can be derived as 
\begin{equation} 
    \Vert (u_c)^q \Vert_{\infty} \leq \frac{\sqrt{L}}{2} |(u_c)^q|_1 \leq \frac{\sqrt{L}}{2}\left(|U^q|_1 + |(e_c)^q|_1\right) \leq \frac{\sqrt{L}}{2}\left(C_0 \sqrt{L} + C(\tau_c^2 + h^4)\right) \leq C.
\end{equation}
Suppose $\left\{ (\tilde{u}_c)^q_p , (\tilde{w}_c)^q_p \right\}$ is the solution of 
\begin{subequations}\label{eq-pbm7}
\begin{numcases}{} 
    \delta^c_t (\tilde{u}_c)_p^{q-\frac{1}{2}} - \mu\delta^c_t (\tilde{w}_c)_p^{q-\frac{1}{2}} +\Psi\left((\tilde{u}_c)_p^{q-\frac{1}{2}},(\tilde{u}_c)_p^{q-\frac{1}{2}}\right) - \frac{h^2}{2}\Psi\left((\tilde{w}_c)_p^{q-\frac{1}{2}},(\tilde{u}_c)_p^{q-\frac{1}{2}}\right)   \notag \\
    \quad+  \Delta_x (\tilde{u}_c)_p^{q-\frac{1}{2}} - \frac{h^2}{6}\Delta_x (\tilde{w}_c)_p^{q-\frac{1}{2}}  - \lambda (\tilde{w}_c)_p^{q-\frac{1}{2}} = r_p^{q-\frac{1}{2}}, \quad 1\leq p\leq M, \ 1\leq q \leq N_c, \label{eq-pbm7a}\\
    (\tilde{w}_c)_p^{q} = \delta_{xx} (\tilde{u}_c)_p^q - \frac{h^2}{12}\delta_{xx} (\tilde{w}_c)_p^q , \quad 1\leq p\leq M, \ 0\leq q \leq N_c, \label{eq-pbm7b} \\
    (\tilde{u}_c)_p^q = (\tilde{u}_c)_{p+M}^q, \quad (\tilde{w}_c)_p^q = (\tilde{w}_c)_{p+M}^q, \quad 1\leq p\leq M, \ 0\leq q \leq N_c, \label{eq-pbm7c} \\
    (\tilde{u}_c)_p^0=\phi\left(x_p\right) +  \zeta\left(x_p\right),\quad 1\leq p\leq M, \label{eq-pbm7d}
\end{numcases}
\end{subequations}
where $r_p^{q-\frac{1}{2}}$ and $\zeta\left(x_p\right)$ denote the perturbation about source term and initial value, respectively.
Denote 
\begin{equation*}
    (\tilde{\varepsilon}_c)_p^{q-\frac{1}{2}} = (\tilde{u}_c)_p^{q-\frac{1}{2}}-(u_c)_p^{q-\frac{1}{2}},\quad  
    (\tilde{\sigma}_c)_p^{q-\frac{1}{2}} = (\tilde{w}_c)_p^{q-\frac{1}{2}}-(w_c)_p^{q-\frac{1}{2}}.
\end{equation*}
Subtracting the system \eqref{eq-pbm5a}-\eqref{eq-pbm5d} from \eqref{eq-pbm7a}-\eqref{eq-pbm7d}, the resulting perturbation system is as follows:
\begin{subequations}\label{eq-ptbpbm1}
\begin{numcases}{} 
    \delta^c_t (\tilde{\varepsilon}_c)_p^{q-\frac{1}{2}} - \mu\delta^c_t (\tilde{\sigma}_c)_p^{q-\frac{1}{2}} +\Psi\left((\tilde{u}_c)_p^{q-\frac{1}{2}},(\tilde{u}_c)_p^{q-\frac{1}{2}}\right) - \Psi\left((u_c)_p^{q-\frac{1}{2}},(u_c)_p^{q-\frac{1}{2}}\right) \notag \\
    \quad - \frac{ h^2}{2}\left[\Psi\left((\tilde{w}_c)_p^{q-\frac{1}{2}},(\tilde{u}_c)_p^{q-\frac{1}{2}}\right)  - \Psi\left((w_c)_p^{q-\frac{1}{2}},(u_c)_p^{q-\frac{1}{2}}\right) \right]  
    +  \Delta_x (\tilde{\varepsilon}_c)_p^{q-\frac{1}{2}} - \frac{h^2}{6}\Delta_x (\tilde{\sigma}_c)_p^{q-\frac{1}{2}} \notag \\ 
    \quad - \lambda (\tilde{\sigma}_c)_p^{q-\frac{1}{2}} = r_p^{q-\frac{1}{2}}, \ 1\leq p\leq M, \ 1\leq q \leq N_c, \label{eq-ptbpbm1a}\\
    (\tilde{\sigma}_c)_p^{q} = \delta_{xx} (\tilde{\varepsilon}_c)_p^q - \frac{h^2}{12}\delta_{xx} (\tilde{\sigma}_c)_p^q , \quad 1\leq p\leq M, \ 0\leq q \leq N_c, \label{eq-ptbpbm1b} \\
    (\tilde{\varepsilon}_c)_p^q = (\tilde{\varepsilon}_c)_{p+M}^q, \quad (\tilde{\sigma}_c)_p^q = (\tilde{\sigma}_c)_{p+M}^q, \quad 1\leq p\leq M, \ 0\leq q \leq N_c, \label{eq-ptbpbm1c} \\
    (\tilde{\varepsilon}_c)_p^0=\zeta\left(x_p\right),\quad 1\leq p\leq M. \label{eq-ptbpbm1d}
\end{numcases}
\end{subequations}

In Step \Rmnum{3}, we can also obtain $\Vert u^k \Vert_{\infty} \leq \widehat{C}$. 
Similar to \eqref{eq-pbm7a}-\eqref{eq-pbm7d}, a difference scheme can be also established,  which concerns initial value perturbation $\zeta\left(x_p\right)$ and source perturbation $\hat{r}_p^{k-\frac{1}{2}}$ for $1\leq p\leq M$, $0\leq k \leq N_f$. If the solution to this scheme is set as $\left\{ \tilde{u}_p^k , \tilde{w}_p^k \right\}$, and denote $\tilde{\varepsilon}_p^{k-\frac{1}{2}} = \tilde{u}_p^k - u_p^k$, $\tilde{\sigma}_p^{k-\frac{1}{2}} = \tilde{w}_p^k - w_p^k$, then a perturbation system analogous to \eqref{eq-ptbpbm1a}-\eqref{eq-ptbpbm1d} is obtained, where $\left\{ u_p^k , w_p^k \right\}$ is the solution of problem \eqref{eq-pbm6a}-\eqref{eq-pbm6d}.

By taking an inner product of \eqref{eq-ptbpbm1a} with $(\tilde{\varepsilon}_c)^{q-\frac{1}{2}}$, and using the similar proof of Theorem \ref{thm-1}, we can derive the stability of Step \Rmnum{1}, and the stability of Step \Rmnum{3} can also be easily yielded.

\begin{thm}\label{thm-3}
Suppose $\{(\tilde{\varepsilon}_c)_p^{q},(\tilde{\sigma}_c)_p^{q}\}$ is the solution of perturbation problem \eqref{eq-ptbpbm1a}-\eqref{eq-ptbpbm1d}, and  $\{\tilde{\varepsilon}_p^{k},\tilde{\sigma}_p^{k}\}$ is the solution of perturbation problem corresponding to Step \Rmnum{3} .Then we have 
\begin{equation}\label{eq-stb1}
    \Vert (\tilde{\varepsilon}_c)^{q} \Vert^2 \leq C \left( \sum_{n=1}^{q}\Vert r^{n-\frac{1}{2}}\Vert^2 + \Vert \zeta \Vert^2  \right), \quad 0\leq q \leq N_c,
\end{equation}
and 
\begin{equation}
      \Vert \tilde{\varepsilon}^{k} \Vert^2 \leq C \left( \sum_{n=1}^{k}\Vert \hat{r}^{n-\frac{1}{2}}\Vert^2 + \Vert \zeta \Vert^2  \right), \quad 0\leq k \leq N_f.
\end{equation}
\end{thm}

\section{Numerical experiments}\label{sec7}
In this section, we will provide some numerical experiments to confirm the effectiveness of TTCD scheme \eqref{eq-pbm5}-\eqref{eq-pbm6}, and choose the standard NCD scheme \eqref{eq-pbm4a}-\eqref{eq-pbm4d} on the fine grid as the reference.
All examples are implemented on a computer with Intel Core i7-12700 CPU and 16GB of RAM,
the software environment is MATLAB R2023a. 

Assume $U_p^k$ is the exact solution of problem \eqref{eq-bbmb}-\eqref{eq-initcond} at point $(x_p,t_k)$, and $u_p^k$ is the numerical solution obtained by the scheme \eqref{eq-pbm6a}-\eqref{eq-pbm6d}.
To evaluate the convergence of the proposed scheme, we define the maximum norm error of the TTCD scheme as follows:
$ \mathrm{Error}_{\infty}^{\mathrm{\scriptscriptstyle TTCD}}\left(\tau_f,h\right)=\max_{\substack{1\leq p\leq M,\\ 0\leq k\leq N_f}} \left|u_{p}^{k}-U_{p}^{k}\right|,$
and
\begin{equation*}
    \mathrm{Rate}^{t}_{\mathrm{\scriptscriptstyle TTCD}} = \log_2\left(\frac{\mathrm{Error}_{\infty}^{\mathrm{\scriptscriptstyle TTCD}}\left(2\tau_f,h\right)}{\mathrm{Error}_{\infty}^{\mathrm{\scriptscriptstyle TTCD}}\left(\tau_f,h\right)}\right), \quad 
    \mathrm{Rate}^{s}_{\mathrm{\scriptscriptstyle TTCD}} = \log_2\left(\frac{\mathrm{Error}_{\infty}^{\mathrm{\scriptscriptstyle TTCD}}\left(\tau_f,2h\right)}{\mathrm{Error}_{\infty}^{\mathrm{\scriptscriptstyle TTCD}}\left(\tau_f,h\right)}\right).
\end{equation*} 
where $\mathrm{Rate}^{t}$ and $\mathrm{Rate}^{s}$ denote the temporal and spatial convergence rates, respectively. For the NCD scheme, the notations $\mathrm{Error}_{\infty}^{\mathrm{\scriptscriptstyle NCD}}$, $\mathrm{Rate}^{t}_{\mathrm{\scriptscriptstyle NCD}}$ and $\mathrm{Rate}^{s}_{\mathrm{\scriptscriptstyle NCD}}$ can be defined similarly. If the exact solution is unknown, the errors and convergence rates are assessed as follows:
\begin{equation*}
    \begin{aligned}
        &\mathrm{Error}_{\infty,t}^{\mathrm{\scriptscriptstyle TTCD}}\left(\tau_f,h\right)=\max_{\substack{1\leq p\leq M,\\ 0\leq k\leq N_f}} \left|u_{p}^{k}\left(\tau_f,h\right)-u_{p}^{2k}\left(\frac{\tau_f}{2},h\right)\right|,  \quad \mathrm{Rate}^{t,\star}_{\mathrm{\scriptscriptstyle TTCD}} = \log_2\left(\frac{\mathrm{Error}_{\infty,t}^{\mathrm{\scriptscriptstyle TTCD}}\left(2\tau_f,h\right)}{\mathrm{Error}_{\infty,t}^{\mathrm{\scriptscriptstyle TTCD}}\left(\tau_f,h\right)}\right), \\
        &\mathrm{Error}_{\infty,s}^{\mathrm{\scriptscriptstyle TTCD}}\left(\tau_f,h\right)=\max_{\substack{1\leq p\leq M,\\ 0\leq k\leq N_f}} \left|u_{p}^{k}\left(\tau_f,h\right)-u_{2p}^{k}\left(\tau_f,\frac{h}{2}\right)\right|,  \quad \mathrm{Rate}^{s,\star}_{\mathrm{\scriptscriptstyle TTCD}} = \log_2\left(\frac{\mathrm{Error}_{\infty,s}^{\mathrm{\scriptscriptstyle TTCD}}\left(\tau_f,2h\right)}{\mathrm{Error}_{\infty,s}^{\mathrm{\scriptscriptstyle TTCD}}\left(\tau_f,h\right)}\right).        
    \end{aligned}
\end{equation*}
Then we define $\mathrm{Error}_{\infty,t}^{\mathrm{\scriptscriptstyle NCD}}\left(\tau_f,h\right)$, $\mathrm{Error}_{\infty,s}^{\mathrm{\scriptscriptstyle NCD}}\left(\tau_f,h\right)$, $\mathrm{Rate}^{t,\star}_{\mathrm{\scriptscriptstyle NCD}}$ , $\mathrm{Rate}^{s,\star}_{\mathrm{\scriptscriptstyle NCD}}$ as the same way.

In this work, we use the fixed point iteration to solve the NCD scheme \eqref{eq-pbm4a}-\eqref{eq-pbm4d}.
The concrete iteration step as follows:
\begin{small}
\begin{equation*}
    \left\{
    \begin{aligned}
        &\frac{1}{\tau_{\eta}}\left( \textbf{u}^{l,m+1}-\textbf{u}^{l-1} \right)-\frac{\mu}{\tau_{\eta}}\left( \textbf{w}^{l,m+1}-\textbf{w}^{l-1} \right)+ \Psi\left(\textbf{u}^{l-\frac{1}{2}},\textbf{u}^{l-\frac{1}{2}}\right) - \frac{h^2}{2}\Psi\left(\textbf{w}^{l-\frac{1}{2}},\textbf{u}^{l-\frac{1}{2}}\right)  \\
        &\quad + \frac{1}{2} \left(\Delta_x \textbf{u}^{l,m+1} +   \Delta_x \textbf{u}^{l-1}\right) - \frac{h^2}{12}\left(\Delta_x \textbf{w}^{l,m+1} + \Delta_x \textbf{w}^{l-1} \right)-\frac{\lambda}{2}\left( \textbf{w}^{l,m+1} + \textbf{w}^{l-1}\right) = 0, \\
        &\Psi\left(\textbf{u}^{l-\frac{1}{2}},\textbf{u}^{l-\frac{1}{2}}\right) = \frac{1}{4}\left(\Psi\left(\textbf{u}^{l,m+1},\textbf{u}^{l,m}\right) + \Psi\left(\textbf{u}^{l,m+1},\textbf{u}^{l-1}\right) + \Psi\left(\textbf{u}^{l-1},\textbf{u}^{l,m+1}\right) + \Psi\left(\textbf{u}^{l-1},\textbf{u}^{l-1}\right)  \right), \\
        &\Psi\left(\textbf{w}^{l-\frac{1}{2}},\textbf{u}^{l-\frac{1}{2}}\right) = \frac{1}{4}\left(\Psi\left(\textbf{w}^{l,m+1},\textbf{u}^{l,m}\right) + \Psi\left(\textbf{w}^{l,m+1},\textbf{u}^{l-1}\right) + \Psi\left(\textbf{w}^{l-1},\textbf{u}^{l,m+1}\right) + \Psi\left(\textbf{w}^{l-1},\textbf{u}^{l-1}\right)  \right), \\
        &\textbf{w}^{l} = \left(\mathcal{I} + \frac{h^2}{12}\delta_{xx}\right)^{-1}\delta_{xx}\textbf{u}^{l}, \quad 1\leq l\leq N_{\eta}, \quad \eta = c,f,
    \end{aligned}
    \right.
\end{equation*}
\end{small}
where $\textbf{u}^{l}=(u^{l}_1,u^{l}_2,\cdots,u^{l}_M)^T$ and $\textbf{w}^{l}=(w^{l}_1,w^{l}_2,\cdots,w^{k}_M)^T$. Here, the number $m$ denotes the $m$-th iteration, $\mathcal{I}$ is a identity operator, and the initial iteration value is $\textbf{u}^{l,0}=\textbf{u}^{l-1}$. In each step, we just need to satisfy the following one of stop conditions
\begin{itemize}
\item Stop tolerance:  $\max\limits_{1\leq p\leq M}|\textbf{u}^{l,m+1}_p - \textbf{u}^{l,m}_p|\leq 10^{-12}$ .
\item Maximum iteration number: 200.
\end{itemize}
Note that $\eta$ denotes the time grid type. According to Remark \ref{rmk-1}, the iteration scheme of NCD scheme on the coarse or fine time grid can be readily obtained.

\begin{exm}[Manufactured Solution]\label{exm-1}
    we consider the equation \eqref{eq-bbmb} possess the exact solution $u(x,t) = e^{t}\sin(\pi x)$, $(x,t)\in[0,2]\times [0,1]$ and the period is $L=2$. Then its source term can be derived as
    \begin{equation}\label{eq-exm1source}
        f(x,t) = \left[1+(\mu+\lambda)\pi^2\right]e^{t}\sin(\pi x) + \frac{\pi}{2}e^{2t}\sin(2\pi x) + \pi e^{t} \cos(\pi x).
    \end{equation}
\end{exm}

In Tables \ref{tab:exm1-timeorder-a} and \ref{tab:exm1-timeorder-b}, we show the discrete maximum norm errors, convergence orders in time and CPU times for the TTCD scheme and standard NCD scheme, where the parameter $(\mu,\lambda)$ is chosen as $(1,1)$ and $(1,0.01)$. The numerical results display that both methods achieve second-order convergence in time, which is consistent with the theoretical analysis. Although the errors of the two methods are of the same order of magnitude,the computational cost of TTCD scheme is obviously less than that of NCD scheme. Especially for the large temporal step-size ratio $\beta_{\tau}$, TTCD scheme can save more running time.

In Table \ref{tab:exm1-spaceorder}, with the fixed coarse grid step $\tau_c=1/2500$ and $\beta_{\tau}=4$, we find that both schemes reach theoretical fourth-order convergence in space for various parameter $(\mu,\lambda)$, with errors that are almost identical. In particular, TTCD scheme clearly saves on calculation costs. Figure \ref{fig-exm1} shows the convergence rate of TTCD scheme for the different parameter $(\mu,\lambda)$, where the second-order temporal  and fourth-order spatial convergence can be observed intuitively. Figure \ref{fig-exm1-cpu} displays that the maximum norm error decreases as the CPU time increases for both schemes. Howerver, under the same accuracy, TTCD scheme requires less CPU time than NCD scheme.
In short, we can conclude  that TTCD scheme reduces the computational cost while maintaining the accuracy of NCD scheme.

\begin{table}[htbp]
\renewcommand{\arraystretch}{1.1}
\centering
\caption{The Maximum norm errors, temporal convergence rates and CPU times for Example \ref{exm-1} with $h=1/600$ and $\beta_{\tau}=2$.}
\label{tab:exm1-timeorder-a}
\resizebox{\textwidth}{!}{
\begin{tabular}{cccccccccc}
\toprule
\multirow{1}{*}{$(\mu,\lambda)$}
&$\tau_c$ & $\tau_f$ & $\mathrm{Error}_{\infty}^{\mathrm{\scriptscriptstyle TTCD}}$ & $\mathrm{Rate}^t_{\mathrm{\scriptscriptstyle TTCD}}$ & CPU(s) & $\mathrm{Error}_{\infty}^{\mathrm{\scriptscriptstyle NCD}}$ & $\mathrm{Rate}^t_{\mathrm{\scriptscriptstyle NCD}}$ & CPU(s) \\
\midrule
\multirow{4}{*}{$(1,1)$}
&1/8 & 1/16 & $4.5684 \times 10^{-04}$ & - & 10.59 & $4.3088 \times 10^{-04}$ & - & 14.90 \\

& 1/16  & 1/32 & $1.1421 \times 10^{-04}$ & 2.0000 & 19.81 & $1.0774 \times 10^{-04}$ & 1.9997 & 25.36 \\

&1/32  & 1/64 & $2.8551 \times 10^{-05}$ & 2.0000 & 36.46 & $2.6935 \times 10^{-05}$ & 2.0000 & 51.27 \\

&1/64 & 1/128 & $7.1381 \times 10^{-06}$ & 1.9999 & 70.32 & $6.7321 \times 10^{-06}$ & 2.0004 & 90.02 \\
\midrule
\multirow{4}{*}{$(1,0.01)$}
& 1/8 & 1/16 & $6.5858 \times 10^{-04}$ & - & 10.62 & $6.0428 \times 10^{-04}$ & - & 16.25 \\

& 1/16  & 1/32 & $1.6448 \times 10^{-04}$ & 2.0014 & 19.82 & $1.5108 \times 10^{-04}$ & 1.9999 & 25.95 \\

&1/32  & 1/64 & $4.1111 \times 10^{-05}$ & 2.0003 & 36.43 & $3.7772 \times 10^{-05}$ & 2.0000 & 48.34 \\

&1/64 & 1/128 & $1.0274 \times 10^{-05}$ & 2.0005 & 70.40 & $9.4449 \times 10^{-06}$ & 1.9997 & 85.74 \\
\bottomrule
\end{tabular}
}
\end{table}

\begin{table}[htbp]
\renewcommand{\arraystretch}{1.1}
\centering
\caption{The Maximum norm errors, temporal convergence rates and CPU times for Example \ref{exm-1} with $h=1/500$ and $\beta_{\tau}=3$.}
\label{tab:exm1-timeorder-b}
\resizebox{\textwidth}{!}{
\begin{tabular}{cccccccccc}
\toprule
\multirow{1}{*}{$(\mu,\lambda)$}
&$\tau_c$ & $\tau_f$ & $\mathrm{Error}_{\infty}^{\mathrm{\scriptscriptstyle TTCD}}$ & $\mathrm{Rate}^t_{\mathrm{\scriptscriptstyle TTCD}}$ & CPU(s) & $\mathrm{Error}_{\infty}^{\mathrm{\scriptscriptstyle NCD}}$ & $\mathrm{Rate}^t_{\mathrm{\scriptscriptstyle NCD}}$ & CPU(s) \\
\midrule
\multirow{4}{*}{$(1,1)$}
&1/6 & 1/18 & $5.4253 \times 10^{-04}$ & - & 7.42 & $3.4047 \times 10^{-04}$ & - & 12.68 \\

& 1/12  & 1/36 & $1.3568 \times 10^{-04}$ & 1.9995 & 13.29 & $8.5128 \times 10^{-05}$ & 1.9998 & 23.65 \\

&1/24  & 1/72 & $3.3920 \times 10^{-05}$ & 2.0000 & 25.11 & $2.1285 \times 10^{-05}$ & 1.9998 & 44.74 \\

&1/48 & 1/144 & $8.4842 \times 10^{-06}$ & 1.9993 & 47.38 & $5.3245 \times 10^{-06}$ & 1.9991 & 76.89 \\
\midrule
\multirow{4}{*}{$(1,0.01)$}
&1/6 & 1/18 & $7.9326 \times 10^{-04}$ & - & 7.47 & $4.7747 \times 10^{-04}$ & - & 12.36 \\

& 1/12  & 1/36& $1.9773 \times 10^{-04}$ & 2.0043 & 13.42 & $1.1937 \times 10^{-04}$ & 1.9999 & 21.86 \\

&1/24  & 1/72 & $4.9396 \times 10^{-05}$ & 2.0011 & 25.38 & $2.9847 \times 10^{-05}$ & 1.9998 & 41.80 \\

&1/48 & 1/144 & $1.2344 \times 10^{-05}$ & 2.0006 & 47.66 & $7.4554 \times 10^{-06}$ & 2.0012 & 72.75 \\
\bottomrule
\end{tabular}
}
\end{table}

\begin{table}[htbp]
\renewcommand{\arraystretch}{1.1}
\centering
\caption{The Maximum norm errors and spatial convergence rates for Example \ref{exm-1} with $\tau_c=1/2500$ and $\beta_{\tau}=4$.}
\label{tab:exm1-spaceorder}
\begin{tabular}{cccccccccc}
\toprule
\multirow{1}{*}{$(\mu,\lambda)$}
& $h$ & $\mathrm{Error}_{\infty}^{\mathrm{\scriptscriptstyle TTCD}}$ & $\mathrm{Rate}^s_{\mathrm{\scriptscriptstyle TTCD}}$& CPU(s)  & $\mathrm{Error}_{\infty}^{\mathrm{\scriptscriptstyle NCD}}$ & $\mathrm{Rate}^s_{\mathrm{\scriptscriptstyle NCD}}$ & CPU(s) \\
\midrule
\multirow{4}{*}{$(1,1)$}
& 1/6 & $1.8117 \times 10^{-03}$ & - & 2.33 & $1.8117 \times 10^{-03}$ & - & 4.21 \\

& 1/12 & $1.1796 \times 10^{-04}$ & 3.9410 & 3.34 & $1.1796 \times 10^{-04}$ & 3.9410 & 5.91 \\

& 1/24 & $7.5077 \times 10^{-06}$ & 3.9738 & 7.53 & $7.5087 \times 10^{-06}$ & 3.9736 & 12.28 \\

& 1/48 & $4.6693 \times 10^{-07}$ & 4.0071 & 21.54 & $4.6800 \times 10^{-07}$ & 4.0040 & 35.70  \\
\midrule
\multirow{4}{*}{$(1,0.01)$}
& 1/6 & $2.1496 \times 10^{-03}$ & -  & 2.35 & $2.1496 \times 10^{-03}$ & - & 4.16 \\

& 1/12 & $1.4475 \times 10^{-04}$ & 3.8924 & 3.28 & $1.4475 \times 10^{-04}$ & 3.8924 & 5.75 \\

& 1/24 & $9.1123 \times 10^{-06}$ & 3.9896 & 7.25 & $9.1139 \times 10^{-06}$ & 3.9894 & 12.09 \\

& 1/48 & $5.7014\times 10^{-07}$ & 3.9984 & 21.68 & $5.7188 \times 10^{-07}$ & 3.9943 & 36.45 \\
\bottomrule
\end{tabular}
\end{table}

\begin{figure}[htpb]
  \centering
  \subfloat[$h=\frac{1}{500}$ and $\beta_{\tau}=3$.]{
      \label{fig-exm1_a}   
      \includegraphics[width=0.4\textwidth]{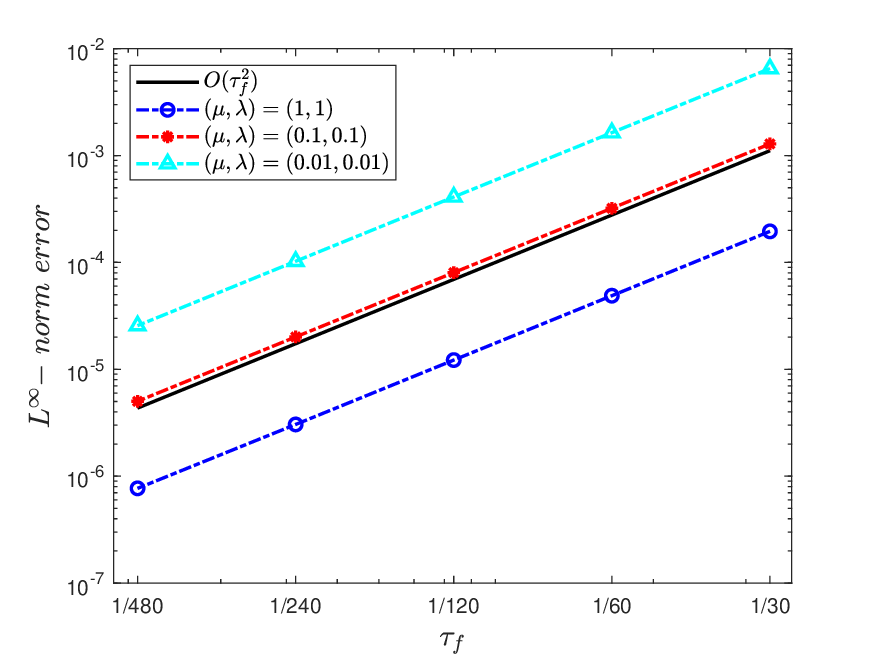}}
  \subfloat[$\tau_f=\frac{1}{6000}$ and $\beta_{\tau}=3$.]{
      \label{fig-exm1_b}
      \includegraphics[width=0.4\textwidth]{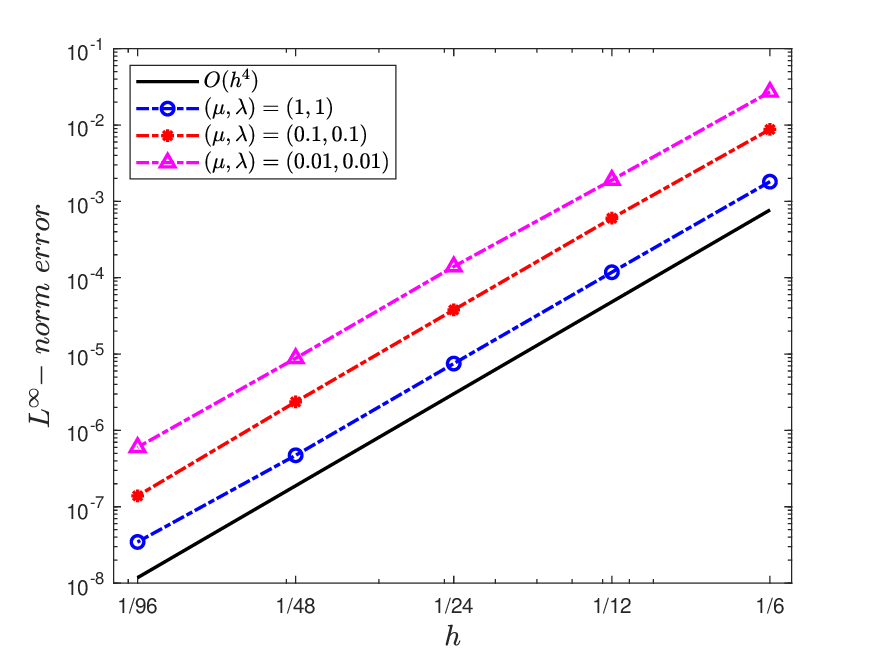}}
  \caption{The Maximum norm error of TTCD scheme for Example \ref{exm-1}.}
  \label{fig-exm1}
\end{figure}

\begin{figure}[htpb]
  \centering 
  \includegraphics[width=0.4\textwidth]{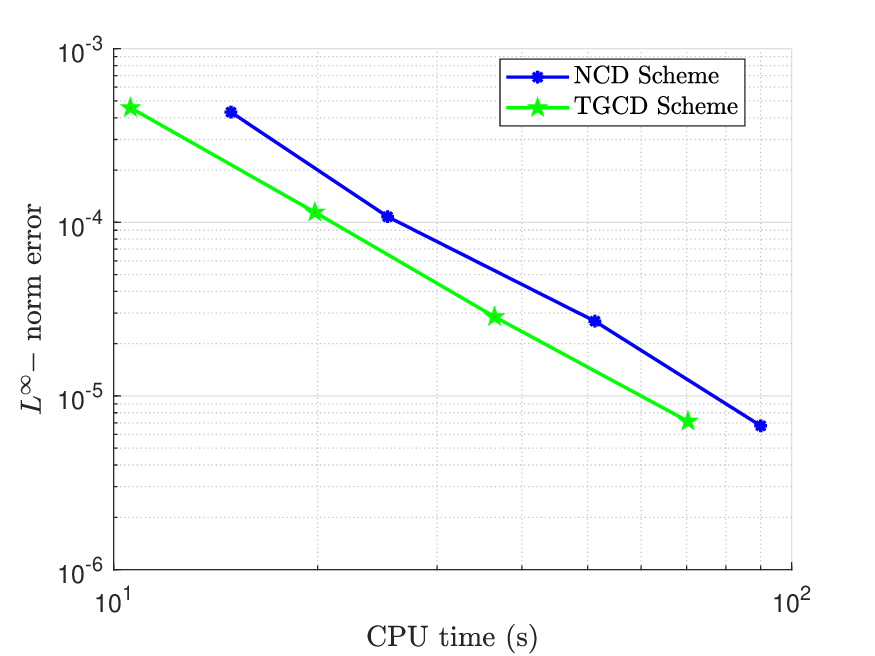}
  \caption{The CPU time for Example \ref{exm-1} with $(\mu,\lambda)=(1,1)$, $h=1/600$ and $\beta_{\tau}=2$.}
  \label{fig-exm1-cpu}
\end{figure}

\begin{exm}[Unknown Solution]\label{exm-2}
    In this example, we consider the following initial value condition 
    \begin{equation}\label{eq-exminitial1}
        u(x,0) = \frac{\sqrt{6}}{3} \mathrm{sech}^2\left(\frac{x}{3}\right),
    \end{equation}
    where the domain is $(x,t)\in[-30,30]\times[0,1]$ and the period is $L=60$. There has no exact solution here.
\end{exm}

\begin{table}[htbp]
\renewcommand{\arraystretch}{1.1}
\centering
\caption{The Maximum norm errors, temporal convergence rates and CPU times for Example \ref{exm-2} with $h=1/20$ and $\beta_{\tau}=4$.}
\label{tab:exm2-timeorder-b}
\resizebox{\textwidth}{!}{
\begin{tabular}{cccccccccc}
\toprule
\multirow{1}{*}{$(\mu,\lambda)$}
&$\tau_c$ & $\tau_f$ & $\mathrm{Error}_{\infty,t}^{\mathrm{\scriptscriptstyle TTCD}}$ & $\mathrm{Rate}^{t,\star}_{\mathrm{\scriptscriptstyle TTCD}}$ & CPU(s) & $\mathrm{Error}_{\infty,t}^{\mathrm{\scriptscriptstyle NCD}}$ & $\mathrm{Rate}^{t,\star}_{\mathrm{\scriptscriptstyle NCD}}$ & CPU(s) \\
\midrule
\multirow{4}{*}{$(1,1)$}
&1/10 & 1/40 & $1.0033 \times 10^{-05}$ & - & 15.83 & $4.4675 \times 10^{-06}$ & - & 30.14 \\

& 1/20  & 1/80 & $2.5088 \times 10^{-06}$ & 1.9997 & 31.61 & $1.1169 \times 10^{-06}$ & 1.9999 & 60.34 \\

&1/40  & 1/160 & $6.2720 \times 10^{-07}$ & 2.0000 & 57.68 & $2.7919 \times 10^{-07}$ & 2.0002 & 99.26 \\

&1/80 & 1/320 & $1.5684 \times 10^{-07}$ & 1.9996 & 115.85 & $6.9852 \times 10^{-08}$ & 1.9989 & 215.16 \\
\midrule
\multirow{4}{*}{$(1,0.01)$}
&1/10 & 1/40 & $1.5300 \times 10^{-05}$ & - & 16.01 & $4.4377 \times 10^{-06}$ & - & 30.69 \\

& 1/20  & 1/80 & $3.8278 \times 10^{-06}$ & 1.9989 & 31.97 & $1.1095 \times 10^{-06}$ & 1.9999 & 61.31 \\

&1/40  & 1/160 & $9.5713 \times 10^{-07}$ & 1.9997 & 57.84 & $2.7738 \times 10^{-07}$ & 1.9999 & 99.71 \\

&1/80 & 1/320 & $2.3923 \times 10^{-07}$ & 2.0003 & 116.33 & $6.9276 \times 10^{-08}$ & 2.0014 & 206.70 \\
\bottomrule
\end{tabular}
}
\end{table}

\begin{table}[htbp]
\renewcommand{\arraystretch}{1.1}
\centering
\caption{The Maximum norm errors and spatial convergence rates for Example \ref{exm-2} with $\tau_c=1/1000$ and $\beta_{\tau}=4$.}
\label{tab:exm2-spaceorder}
\begin{tabular}{cccccccccc}
\toprule
\multirow{1}{*}{$(\mu,\lambda)$}
& $h$ & $\mathrm{Error}_{\infty,s}^{\mathrm{\scriptscriptstyle TTCD}}$ & $\mathrm{Rate}^{s,\star}_{\mathrm{\scriptscriptstyle TTCD}}$& CPU(s)  & $\mathrm{Error}_{\infty,s}^{\mathrm{\scriptscriptstyle NCD}}$ & $\mathrm{Rate}^{s,\star}_{\mathrm{\scriptscriptstyle NCD}}$ & CPU(s) \\
\midrule
\multirow{4}{*}{$(1,1)$}
& 3/4 & $3.1652 \times 10^{-04}$ & - & 6.81 & $3.1652 \times 10^{-04}$ & - & 10.56 \\

& 3/8 & $2.2576 \times 10^{-05}$ & 3.8094 & 18.92 & $2.2576 \times 10^{-05}$ & 3.8095 & 28.30 \\

& 3/16 & $1.4247 \times 10^{-06}$ & 3.9860 & 65.48 & $1.4247 \times 10^{-06}$ & 3.9860 & 96.78 \\

& 3/32 & $8.9877 \times 10^{-08}$ & 3.9866 & 526.54 & $8.9848 \times 10^{-08}$ & 3.9871 & 936.86  \\
\midrule
\multirow{4}{*}{$(1,0.01)$}
& 3/4 & $5.4130 \times 10^{-04}$ & -  & 7.50 & $5.4130 \times 10^{-04}$ & - & 10.84 \\

& 3/8 & $3.9436 \times 10^{-05}$ & 3.7788 & 20.05 & $3.9436 \times 10^{-05}$ & 3.7788 & 29.44 \\

& 3/16 & $2.4927 \times 10^{-06}$ & 3.9837 & 69.66 & $2.4927 \times 10^{-06}$ & 3.9837 & 101.83 \\

& 3/32 & $1.5695 \times 10^{-07}$ & 3.9894 & 691.70 & $1.5698 \times 10^{-07}$ & 3.9891 & 1071.92 \\
\bottomrule
\end{tabular}
\end{table}

\begin{table}[htbp]
\renewcommand{\arraystretch}{1.1}
\centering
\setlength\tabcolsep{18pt}{ 
\caption{For Example \ref{exm-2}, the numerical conservation invariant $E^k$ defined in \eqref{eq-numinvariant} with mesh size $h=1/10$ and $\tau_f=1/1024$. }
\label{tab:exm2-invariant}
\begin{tabular}{ccccc}
\toprule
$t$ & $(\mu,\lambda)=(1,1)$ & $(\mu,\lambda)=(0.1,0.1)$ & $(\mu,\lambda)=(0.01,0.01)$ \\
\midrule
0 & 2.903703684187 & 2.690370368419 & 2.669037036842\\
1 & 2.903703684212 & 2.690370368490 & 2.669037036927\\
2 & 2.903703684227 & 2.690370368653 & 2.669037037266\\
3 & 2.903703684230 & 2.690370368908 & 2.669037038386\\
4 & 2.903703684220 & 2.690370369262 & 2.669037042923\\
5 & 2.903703684196 & 2.690370369724 & 2.669037062762\\
\bottomrule
\end{tabular}}
\end{table}

\begin{figure}[htpb]
  \centering
  \subfloat[$h=\frac{1}{15}$ and $\beta_{\tau}=3$.]{
      \label{fig-exm2_a}   
      \includegraphics[width=0.4\textwidth]{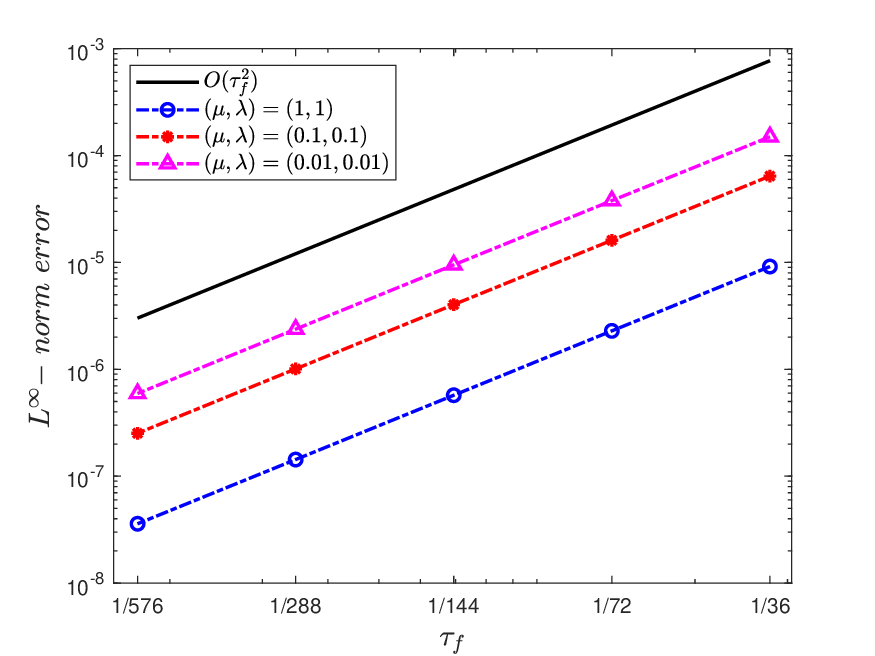}}
  \subfloat[$\tau_f=\frac{1}{1500}$ and $\beta_{\tau}=3$.]{
      \label{fig-exm2_b}
      \includegraphics[width=0.4\textwidth]{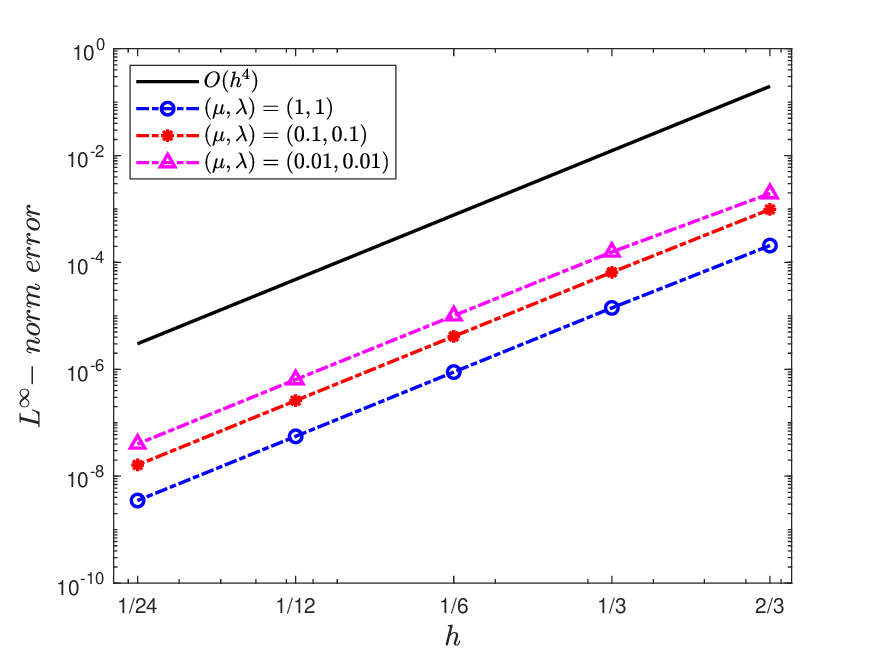}}
  \caption{The Maximum norm error of TTCD scheme for Example \ref{exm-2}.}
  \label{fig-exm2}
\end{figure}
\begin{figure}[htpb]
  \centering 
  \includegraphics[width=0.4\textwidth]{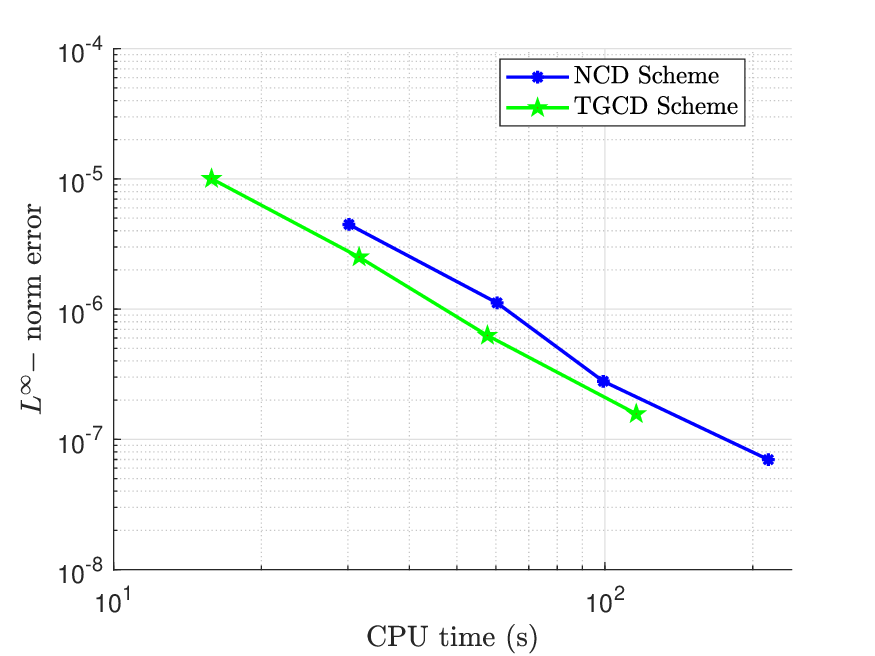}
  \caption{The CPU time for Example \ref{exm-2} with $(\mu,\lambda)=(1,1)$, $h=1/20$ and $\beta_{\tau}=4$.}
  \label{fig-exm2-cpu}
\end{figure}

By observing the Tables \ref{tab:exm2-timeorder-b} and \ref{tab:exm2-spaceorder}, we can see that both schemes achieve the temporal second-order and spatial fourth-order convergence, which verifies the theoretical results further. 
Note that the period is $L=60$ in this example, which means that, compared with Example \ref{exm-1}, we need more space points under the same space step size $h$. Comparing the CPU times of this two methods in Tables \ref{tab:exm2-timeorder-b} and \ref{tab:exm2-spaceorder}, the NCD scheme
is extremely time-consuming, and TTCD scheme can significantly reduce the running time without loss of accuracy. 

In Table \ref{tab:exm2-invariant}, we list the computed values of conservation invariant $E^k$ defined in \eqref{eq-numinvariant}. The results display that TTCD scheme can maintain
the conservation invariants approximately under various parameters $(\mu, \lambda)$, even for the larger space step size $h$.
Figure \ref{fig-exm2} clearly displays the sencond-order temporal and fourth-order spatial convergence of the TTCD scheme, even for the various parameter $(\mu,\lambda)$. In addition, Figure \ref{fig-exm2-cpu} presents the relationship between the maximum norm errors and CPU times, reflecting the fact that the TTCD scheme is more efficient than the NCD scheme.

\section{Conclusion}\label{sec8}
We have proposed a Temporal Two-Grid Compact Difference (TTCD) scheme based on the Crank-Nicolson method and a temporal two-grid algorithm for solving the one-dimensional BBMB equation. Using the energy method, the scheme was shown to achieve second-order accuracy in time and fourth-order accuracy in space in the maximum norm. Furthermore, the conservation property, unique solvability, and stability of the scheme have been rigorously established. Finally, two numerical examples confirm the theoretical results. The comparison of computational outcomes demonstrates that the TTCD scheme reduces CPU time while maintaining the same accuracy as the standard nonlinear compact difference scheme.

\bibliographystyle{alpha}
\bibliography{references}

\end{document}